\documentclass[12pt]{amsart}
\usepackage{amsmath}
\usepackage{amssymb}
\usepackage{tabularx}
\usepackage{enumerate}
\usepackage{graphicx}
\usepackage{texdraw}
\usepackage{exscale}
\usepackage{relsize}
\usepackage{mathrsfs}
\usepackage{latexsym}
\usepackage{amsfonts,amssymb,amsmath}
\usepackage{epsfig}
\usepackage{color}

\makeatletter
\@addtoreset{equation}{section}

\makeatother
\newtheorem{thm}{Theorem}[section]
\newtheorem{lem}[thm]{Lemma}
\newtheorem{cor}[thm]{Corollary}
\newtheorem{prop}[thm]{Proposition}
\newtheorem{remark}[thm]{Remark}

\newcommand{\R}{\mathbb{R}}

\begin{document}
\title[A reaction-diffusion model with Allee effect and protection zone]
{The role of protection zone on species spreading governed by a reaction-diffusion model with strong Allee effect$^\S$}

 \thanks{$\S$  K. Du was partially supported by NSF of China (No. 11801084), 
 R. Peng was partially supported by NSF of China (Nos. 11671175 and 11571200), the Priority Academic Program Development of Jiangsu Higher Education Institutions, Top-notch Academic Programs Project of Jiangsu Higher Education Institutions (No. PPZY2015A013) and Qing Lan Project of Jiangsu Province, and N. Sun was partially supported by NSF of China (No. 11801330) and NSF of Shandong Province (No. ZR2017BA023) of China.}
\author[K. Du, R. Peng,  N. Sun]{Kai Du$^\dag$, Rui Peng$^\ddag$ and Ningkui Sun$^\sharp$}
\thanks{$\dag$  Shanghai Center for Mathematical Sciences, Fudan University, 2005 Songhu Road, Shanghai 200438, China.}
\thanks{$\ddag$ School of Mathematics and Statistics, Jiangsu Normal University, Jiangsu 221116, China.}
\thanks{$\sharp$ School of Mathematical Science, Shandong Normal University, Jinan 250014, China.}
\thanks{E-mails: {\sf kdu@fudan.edu.cn} (Du), {\sf pengrui$\_$seu@163.com} (Peng), {\sf sunnk@sdnu.edu.cn} (Sun)}
\date{}

\begin{abstract}
It is known that a species dies out in the long run for small initial data if its evolution obeys a reaction of bistable nonlinearity. Such a phenomenon, which is termed as the strong Allee effect,
is well supported by numerous evidence from ecosystems, mainly due to the environmental pollution
as well as unregulated harvesting and hunting. To save an endangered species, in this paper we introduce a protection zone that is governed by a Fisher-KPP nonlinearity, and examine the dynamics of a reaction-diffusion model
with strong Allee effect and protection zone. We show the existence of two critical values $0<L_*\leq L^*$, and prove that a vanishing-transition-spreading trichotomy result holds when the length of protection zone is smaller than $L_*$; a transition-spreading dichotomy result holds when the length of protection zone is between $L_*$ and $L^*$; only spreading happens when the length of protection zone is larger than $L^*$. This suggests that the protection zone works when its length is larger than the critical value $L_*$. Furthermore, we compare two types of protection zone with the same length:
a connected one and a separate one, and our results reveal that the former is better for species spreading than the latter.

\end{abstract}

\maketitle

\section{Introduction}\label{sec:pro}
\subsection{Derivation of our model}
It has long been accepted that the dynamics of a biological population is affected by both the environmental factors and the population density \cite{CC}. In the pioneering work \cite{A} in 1931, through experimental studies, W.C. Allee demonstrated that goldfish grow more rapidly when there are more individuals within the tank, which led him to conclude that aggregation can improve the survival rate of individuals. This phenomenon is termed as Allee effect.

Nowadays, Allee effect is broadly defined as a decline in individual fitness at low population size or density, and it can be classified into the strong one and the weak one. A population exhibiting a weak Allee effect will possess a reduced per capita growth rate (directly related to individual fitness of the population) at lower population density or size. However, even at this low population size or density, the population will always exhibit a positive per capita growth rate. Meanwhile, a population exhibiting a strong Allee effect will have a critical population size or density under which the population growth rate becomes negative. Therefore, when the population density or size hits a number below this threshold, the population will be destined for extinction; see \cite{CBG}.

There are a variety of mechanisms that can cause Allee effect, including mating systems, predation, environmental modification, and social interactions among others. For instance, there is an alarming threat to wild life and biodiversity due to the pollution of the environment as well as unregulated harvesting and hunting; see \cite{Av,BD,DDNY,JK} for related discussion.

If one uses a reaction-diffusion model to describe the spatiotemporal evolution of a single species that
is subject to the strong Allee effect, a typical reaction function is the so-called bistable nonlinearity.
Indeed, the role of Allee effect in population spreading and invasion in reaction-diffusion or integro-differential models has been explored in many research works; see \cite{DuS,KLH,La,LKa,ML,SS,SLMNS,WSW,WK,WKN,WSW2} to mention a few.

Now, let us assume that the species lives in an entire line habitat $\R=:(-\infty,\infty)$ and distributes only within a bounded interval initially. Mathematically, this leads us to consider the following Cauchy problem
 \begin{equation}\label{bip}
\left\{
\begin{array}{ll}
 u_t = u_{xx} + g(u), &  t>0, x\in\R,\\
 u(0,x)=u_0(x), & x\in \R,
\end{array}
\right.
\end{equation}
where the unknown function $u(t,x)$ is the density of the species at the location $x$ and time $t$, and the initial datum
$u_0$ is nonnegative and compactly supported.

The nonlinear reaction term $g$ is {\it globally Lipschitz} and satisfies
\begin{equation}\label{bi}
g(0)=g(\theta)= g(1)=0, \quad g(u) \left\{
\begin{array}{l}
<0 \ \ \mbox{in } (0,\theta),\\
>0\ \  \mbox{in } (\theta, 1),\\
< 0\ \ \mbox{in } (1,\infty),
\end{array} \right.
\end{equation}
for some $\theta\in (0,1)$,  $g'(0)<0$, $g'(1)<0$ and
\begin{equation}\label{unbalance}
\int_0^1 g(s) ds >0.
\end{equation}
We call that $g$ is a {\it bistable} nonlinearity. Denote $\theta^*\in(\theta,1)$
to be the constant which is uniquely determined by the condition
\[
\int_0^{\theta^*}g(s)ds=0.
\]

In the sequel, let us first recall the result of the long-time dynamics on \eqref{bip}.
Denote
 $$X=\{\phi\in L^\infty(\R):\ \phi\ \mbox{is nonnegative and compactly supported} \}.$$
Following \cite{DM}, we introduce a one-parameter family of initial data which satisfies

(i)\ $\phi_\sigma\in X$ for every $\sigma>0$ and the map $\sigma\mapsto\phi_\sigma$ is continuous
 from $\R$ to $L^\infty(\R)$;

(ii)\ $\phi_{\sigma_1}\leq,\not\equiv\phi_{\sigma_2}$ in the a.e. sense, for all $0<\sigma_1<\sigma_2$;

(iii)\ $\lim_{\sigma\to0}\phi_\sigma=0$ in $L^\infty(\R)$.

\smallskip Denote by $u_\sigma$ the solution of \eqref{bip} with the initial datum $u_0=\phi_\sigma\, (\sigma>0)$.
By \cite[Theorem 1.3 and Example 1.6]{DM}, we can conclude that there exist $\hat \sigma>0$ and $x_0\in\R$ such that

 \begin{itemize}

\item $\lim_{t\to\infty}u(t,x)=0$ uniformly in $\R$ if $0<\sigma<\hat\sigma;$

\item $\lim_{t\to\infty}u(t,x)=V(x-x_0)$ uniformly in $\R$ if $\sigma=\hat\sigma;$

\item $\lim_{t\to\infty}u(t,x)=1$  locally uniformly in $\R$ if $\sigma>\hat\sigma$.

\end{itemize}
Here, $V$ is a positive symmetrically decreasing solution of
 $$
 V''+f(V)=0\ \ \mbox{in}\ \R,\ \ V(0)=\theta^*,\ \ V'(0)=0.
 $$
Relevant works on \eqref{bip} can be found, for instance, in \cite{K,MX,MX2,P,Zla}.

Roughly speaking, the above conclusion shows that the species dies out in the long run
when the initial data are below a critical value, and the species propagates successfully only when the initial data are larger than the critical value, and hence the strong Allee effect occurs.

Allee effects are very relevant to many conservation programmes, where scientists and managers are often working with populations that have been reduced to low densities or small numbers. Different measures have been taken to protect endangered species and their habitats. Among the effective measures, the approach of providing protected areas has become most popular over the past decades; one may refer to, for instance, \cite{CLMS,CSW,DDNY,DuLiang,DPW,DuS2,DuS3} and the references therein for related studies.

To save the endangered species, in the current paper we introduce a protection zone within which the growth of the species is governed by a Fisher-KPP type nonlinearity. It is widely accepted that the Fisher-KPP type nonlinearity is another fundamental reaction term in describing the evolution of biological population. More precisely, assume that the species density $u$ satisfies
\begin{equation}\label{mop}
\left\{
\begin{array}{ll}
 u_t = u_{xx} + f(u), &  t>0, x\in\R,\\
 u(0,x)=u_0(x), & x\in \R,
\end{array}
\right.
\end{equation}
where the initial datum $u_0\geq,\not\equiv0$ is compactly supported. The nonlinear reaction term $f$ satisfies
\begin{equation}\label{mono}
f(0)=f(1)=0<f'(0), \ \ f'(1)<0,\ \ (1-u)f(u) >0,\ \ \forall u>0, \ u\neq 1.
\end{equation}
We call that $f$ is a {\it monostable} nonlinearity.

It is well known that the solution $u$ of \eqref{mop} has the following long-time behavior \cite{AW1,AW2}:

\begin{itemize}

\item $\lim_{t\to\infty}u(t,x)=1$\ \ locally uniformly in $\R$.

\end{itemize}
This means biologically that the propagation of the species is always successful regardless of its initial stage.

If the species subject to the strong Allee effect is protected in a bounded and connected region,
say $[-L,L]$ for a given $L>0$, so that it follows the Fisher-KPP nonlinear growth therein,
we are led to the following problem
\begin{equation}\label{p-i}
\left\{
\begin{array}{ll}
 u_t = u_{xx} + f(u), &  t>0, -L<x<L,\\
 u_t = u_{xx} + g(u), &  t>0, x\in\R\setminus{[-L,L]},\\
 u(t, -L-0)= u(t, -L+0), & t>0,\\
 u(t, L-0)= u(t, L+0), & t>0,\\
 u_x(t, -L-0)= u_x(t, -L+0), & t>0,\\
 u_x(t, L-0)= u_x(t, L+0), & t>0,\\
 u(0,x)=u_0(x)\geq 0, & x\in \R.
\end{array}
\right.
\end{equation}
Here, for any given $t>0$, $u(t, -L-0)$ and $u_x(t, -L-0)$ represent, respectively,
the left limit value and the left derivative of $u$ with respect to $x$ at $x=-L$, and
$u(t, -L+0)$ and $u_x(t, -L+0)$ are respectively the right limit value and the right derivative of $u$ with respect to $x$ at $x=-L$.

The continuous connection condition $u(t, -L-0)= u(t, -L+0)$ and $u(t, L-0)= u(t, L+0)$ is a natural assumption
at the boundary points $x=-L,L$ of the protection zone. Furthermore, it is necessary to assume the zero flux through the points $x=-L,L$, which is equivalent to require that the fifth and sixth lines in \eqref{p-i} hold.

Given $u_0\in X$, it is known that \eqref{p-i} admits a unique nonnegative solution $u\in C^{1,2}((0,\infty)\times(\R\setminus\{\pm L\}))\cap C^{\alpha/2,1+\alpha}((0,\infty)\times\R)$ for any $\alpha\in(0,1)$, and
$u$ exists for all time $t>0$; refer to \cite{DLPZ,JPS,vonB}. To further simplify \eqref{p-i}, we consider
the scenario that the initial data $u_0$ are symmetric with respect to the origin in the sense that $u_0(x)=u_0(-x)$ a.e. in $\R$.
Clearly, $u(t,x)=u(t,-x)$ for all $t>0,\,x\in\R$, and in turn, $u_x(t,0)=0,\,\forall t>0$. Therefore, it suffices to investigate the following problem:
\begin{equation}\label{p}
\left\{
\begin{array}{ll}
 u_t = u_{xx} + f(u), &  t>0, 0<x<L,\\
 u_t = u_{xx} + g(u), &  t>0, x>L,\\
 u_x(t, 0)= 0, & t>0,\\
 u(t, L-0)= u(t, L+0), & t>0,\\
 u_x(t, L-0)= u_x(t, L+0), & t>0,\\
 u(0,x)=u_0(x)\geq 0, & x\in \R_+,
\end{array}
\right.
\end{equation}
where $\R_+=[0,\infty)$, $L>0$ and the protection zone is $[0,L]$.

\subsection{Statement of our main results}
Our primary goal in this paper is to examine the role of the protection zone by studying the
dynamics of the reaction-diffusion model \eqref{p} with the strong Allee effect.

Throughout the paper, unless otherwise specified, in addition to the previously
imposed conditions \eqref{bi}, \eqref{unbalance} and \eqref{mono} on $f,\,g$, we further assume that

\smallskip

{\bf{(H)}}\ \ The functions $f,\,g$ are {\it globally Lipschitz} and
$g(u)<f(u)\ \mbox{ for all } 0 < u < 1.$

\smallskip
Let us denote
 $$
 X^+=\{\phi\in L^\infty(\R_+):\ \phi\ \mbox{is nonnegative and compactly supported} \}.
 $$
Similarly as in \cite{DM}, we introduce a one-parameter family of initial data
$u_0=\phi_\sigma\, (\sigma>0)$, fulfilling the following conditions:

($\Phi_1$)\ $\phi_\sigma\in X^+$ for every $\sigma>0$ and the map $\sigma\mapsto\phi_\sigma$ is continuous
 from $\R_+$ to $L^\infty(\R_+)$;

($\Phi_2$)\ $\phi_{\sigma_1}\leq,\not\equiv\phi_{\sigma_2}$ in the a.e. sense, for all $0<\sigma_1<\sigma_2$;

($\Phi_3$)\ $\lim_{\sigma\to0}\phi_\sigma=0$ in $L^\infty(\R_+)$.

Given $\phi_\sigma\in X^+$, we define
 $$\hbar=\inf\{h>0:\ \phi_\sigma(x)=0\ \ \mbox{a.e.}\ \forall x\in(h,\infty)\}.$$


\smallskip
For any given $u_0=\phi_\sigma\, (\sigma>0)$, using the comparison principle and classical theory for parabolic equations it is easy to check that problem \eqref{p} admits a unique positive solution $u\in C^{1,2}((0,\infty)\times(\R_+\setminus\{L\}))\cap C^{\alpha/2,1+\alpha}((0,\infty)\times[0,\infty))$ which is uniformly bounded solution with respect to both space and time.
Therefore, one may expect that the long time behavior of solutions will be determined by nonnegative and bounded stationary solutions
of \eqref{p}, that is, the solutions of the following elliptic equation:
\begin{equation}\label{U}
\left\{
\begin{array}{ll}
 U'' + f(U)=0, & 0<x<L,\\
 U'' + g(U)=0, &  x>L,\\
 U'(0)=0,             \\
 U( L-0)= U(L+0), \\
 U'( L-0)= U'(L+0).
\end{array}
\right.
\end{equation}

\smallskip
Now we list some possible situations on the asymptotic behavior of the solutions to \eqref{p}:
\begin{itemize}

\item {\it vanishing} :
$\lim_{t\to\infty}u(t,x)=0$ uniformly in $[0,\infty)$;

\item {\it spreading} : $\lim_{t\to\infty}u(t,x)=1$  locally uniformly in $[0,\infty)$;

\item {\it transition} : $\lim_{t\to\infty}|u(t,x)-U(x)|=0 \mbox{ locally uniformly in $[0,\infty)$}$,
where $U$ is a ground state of \eqref{U}.
\end{itemize}

In the above, by saying a ground state $U$ of \eqref{U},
we mean that $U$ is  a solution to \eqref{U}
satisfying
\[ U(x)>0,\  \ \forall x\in[0,\infty),\ \ \ U'(0)=0,\ \ \ U(+\infty)=0,\]
and when $x>L$, $U(\cdot)=V(\cdot-z)$, where $z\in\R$ and $V$ is the unique positive decreasing solution of
 \[V'' + g(V)=0,\ \ \ V(0)=\theta^*,\ \ \ V(\pm\infty)=0.\]

Denote
 $$
 L_*=\frac{1}{\sqrt{f'(0)}}\arctan\sqrt{-\frac{g'(0)}{f'(0)}}.
 $$
By the proof of the assertion (I) in Theorem \ref{thm:dybe} below, we know that if $0<L<L_*$,
then problem \eqref{U} has a ground state. This allows us to define
\begin{equation}\label{L8}
L^*:=\sup\{L_0 > 0: \ \mbox{ problem \eqref{U} with $L=L_0$ has a ground state}\}.
\end{equation}
In view of Lemmas \ref{lem:LL1} and \ref{propr1} in Section 3, problem \eqref{U} admits a ground state
for any $0<L<L^*$, and $L^*$ can be bounded by
\[
L^*\leq\int_0^{\theta^*}\frac{1}{\sqrt{2\int_r^{\theta^*}f(s)ds}}dr<\infty.
\]

We are now in a position to give a satisfactory description of the long-time dynamical behavior of problem \eqref{p}.

\begin{thm}\label{thm:dybe}
Let $u$ be the solution of \eqref{p} with $u_0=\phi_\sigma\in X^+$,
and $L_*\leq L^*$ be defined as before. The following assertions hold.
\begin{itemize}
\item[\vspace{10pt}(I)] \; (Small protection zone case) If $0<L<L_*$, then there exist $\sigma_*,\ \sigma^*\in(0,\infty)$ with $\sigma_*\leq\sigma^*$
such that the following trichotomy holds:
\begin{itemize}
\item[\vspace{10pt}(i)] \;Vanishing happens when $0<\sigma<\sigma_*;$
\item[\qquad(ii)] \;Transition happens when $\sigma\in[\sigma_*,\sigma^*];$
\item[\qquad(iii)] \;Spreading happens when $\sigma>\sigma^*$.
\end{itemize}

\item[\qquad(II)] \;(Medium-sized protection zone case) If $L_*<L^*$ and $L_*<L<L^*$, then there exists $\sigma^*\in(0,\infty)$
such that the following dichotomy holds:
\begin{itemize}
\item[\vspace{10pt}(i)] \;Transition happens when $\sigma\in(0,\sigma^*];$
\item[\qquad(ii)] \;Spreading happens when $\sigma>\sigma^*$.
\end{itemize}
\item[\qquad(III)] \;(Large protection zone case) If $L>L^*$, then spreading happens for all $\sigma>0$.
\end{itemize}

\end{thm}

Theorem \ref{thm:dybe} indicates that only if the protection zone is suitably long (i.e., $L>L_*$),
the species will survive in the entire space all the time regardless of its initial data.

\begin{remark} Concerning Theorem \ref{thm:dybe}, we would like to make the following comments.

\begin{itemize}
\item[\rm{(i)}] Concrete examples of  $f$ and $g$ satisfying \eqref{bi}, \eqref{unbalance}, \eqref{mono} and (H) can be constructed to show that $L_*<L^*$ holds; on the other hand, numerical simulations have been performed to suggest that $L_*=L^*$ could be possible for certain functions $f$ and $g$.

\item[\rm{(ii)}] It is unclear to us about the long-time behavior of $u$ at the critical values $L_*,\,L^*$; furthermore,
in Theorem \ref{thm:dybe}(II), it would be interesting to know whether $\sigma_*=\sigma^*$ or $\sigma_*<\sigma^*$.

\item[\rm{(iii)}] The above comments also apply to Theorem \ref{thm:dybe1} below.

\end{itemize}

\end{remark}

\smallskip
As mentioned before, the protection zone $[-L,L]$ designed in \eqref{p-i} is a connected interval.
We are now interested in the effect of structure of protection zone on the dynamics. For sake of simplicity, we
consider the situation that a protection zone with the same length $2L$ is designed in a way that it consists of two separate intervals, say $[-L_2,-L_1]$ and $[L_1,L_2]$ with $L_1>0$ and $L_2-L_1=L$. It is also assumed that the initial data $u_0$ are symmetric with respect to the origin. Therefore, as in \eqref{p}, it becomes sufficient to study the following problem:
\begin{equation}\label{q}
\left\{
\begin{array}{ll}
 u_t = u_{xx} + f(u), &  t>0,  x \in(L_1, L_2),\\
 u_t = u_{xx} + g(u), &  t>0, x\in(0, L_1)\cup(L_2,\infty),\\
 u(t, L_i-0)= u(t, L_i+0), & t>0, i=1, 2, \\
 u_x(t,  L_i-0)= u_x(t, L_i+0), & t>0, i=1, 2, \\
 u_x(t,0)=0, & t>0,\\
 u(0,x)=u_0(x)\geq 0, & x\in \R_+,
\end{array}
\right.
\end{equation}
where $L_2 > L_1>0$, and the protection zone is $[L_1, L_2]$.
We always set
\[
L=L_2-L_1,
\]

Let $\tilde{L}_*$ be given in Lemma \ref{lem:1eigenvaluemp2}.
As in the proof of Theorem \ref{thm:dybe}, it can be shown that
the stationary problem (refer to \eqref{q-statio} in Section 2) corresponding to \eqref{q}
has a ground state provided that $0<L<\tilde L_*$. Thus, we can define
\begin{equation}\label{L8-1}
\tilde L^*:=\sup\{L_0 > 0: \ \mbox{ problem \eqref{q-statio} with $L=L_0$ has a ground state}\}.
\end{equation}
Moreover, the same analysis as in that of Lemma \ref{lem:LL1} allows one to conclude that
problem \eqref{q-statio} admits a ground state if $0<L<\tilde L^*$, and
by Proposition \ref{propr1-1} below, we further have
\begin{equation}\label{L8-2}
\tilde L^*\leq2\int_0^{\theta^*}\frac{1}{\sqrt{2\int_r^{\theta^*}f(s)ds}}dr.
\end{equation}

Our result on problem \eqref{q} can be stated as follows.
\begin{thm}\label{thm:dybe1}
Let $u$ be the solution of \eqref{q} with $u_0=\phi_\sigma\in X^+$, and
 $\tilde{L}_*$ and $\tilde{L}^*$ be given in Lemma \ref{lem:1eigenvaluemp2} and \eqref{L8-1}, respectively.
 The following assertions hold.
\begin{itemize}
\item[\vspace{10pt}(I)] \; If $0<{L}<\tilde{L}_*$, then there exist $\tilde\sigma_*,\ \tilde\sigma^*\in(0,\infty)$ with $\tilde\sigma_*\leq\tilde\sigma^*$
such that the following trichotomy holds:
\begin{itemize}
\item[\vspace{10pt}(i)] \;Vanishing happens when $0<\sigma<\tilde\sigma_*;$
\item[\qquad(ii)] \;Transition happens when $\sigma\in[\tilde\sigma_*,\tilde\sigma^*];$
\item[\qquad(iii)] \;Spreading happens when $\sigma>\tilde\sigma^*$.
\end{itemize}

\item[\qquad(II)] \;If $\tilde{L}_*<\tilde{L}^*$ and $\tilde{L}_*<{L}<\tilde{L}^*$, then there exists $\tilde\sigma^*\in(0,\infty)$
such that the following dichotomy holds:
\begin{itemize}
\item[\vspace{10pt}(i)]  \;Transition happens when $\sigma\in(0,\tilde\sigma^*];$
\item[\qquad(ii)]\;Spreading happens when $\sigma>\tilde\sigma^*$.
\end{itemize}

\item[\qquad(III)] \;If ${L}>\tilde{L}^*$, then spreading happens for all $\sigma>0$.
\end{itemize}

\end{thm}

In view of Theorem \ref{thm:dybe1}, we also see that only if the protection zone is suitably long (i.e., $L>\tilde L_*$),
will the species survive in the entire space regardless of its initial data.

Furthermore, by comparing the two types of protection zone designed in \eqref{p} and \eqref{q}, we will find that the connected one is better for species survival than a separate one. More detailed discussion on our results will be given in the final section.

The rest of our paper is organized as follows. In Section 2, we prepare some preliminary results,
including the analysis of the associated stationary solution problems and a general convergence result mainly due to \cite{DM}.
Sections 3 and 4 are devoted to the proof of Theorem \ref{thm:dybe} and Theorem \ref{thm:dybe1}, respectively. In Section 5, we end the paper with the interpretation of the biological implication of our theoretical results.

\section{Some preliminary Results}\label{sec:basic}

In this section, we present some preliminary results which will be frequently used later.

\subsection{Stationary solutions}\label{sub:statin} A stationary solution of \eqref{p} is a solution of
\eqref{U}. Each solution of $q''+f(q)=0$
corresponds to a trajectory $p^2=F(q)-C_1$
in the $qp$-plane, where $p=q'$, $C_1$ is a constant and $F(q)=-2\int_0^qf(v)dv$. Each solution of
$q''+g(q)=0$
corresponds to a trajectory $p^2=G(q)-C_2$ where $C_2$ is a constant and $G(q)=-2\int_0^qg(v)dv$.
Furthermore, for any solution of \eqref{U}, the connection condition at $x= L$ is fulfilled whenever the
trajectory $q'^2=F(q)-C_1$ intersects the trajectory $q'^2=G(q)-C_2$. Note that such an intersection may not be unique, so that several stationary solutions of \eqref{p} can be derived from different trajectories.

It is easy to check that the unique positive symmetrically decreasing solution $V$ of
\[q'' + g(q)=0,\ \ \ q(0)=\theta^*,\ \ \ q(\pm\infty)=0,\]
corresponds to the trajectory $p^2=-2\int_0^qg(v)dv$, which is denoted by $\Gamma_0$; and for $\beta\in(0,1)$,
the unique positive symmetrically decreasing solution $W$ of
\[q'' + f(q)=0,\ \ \ q(0)=\beta,\ \ \ q(\pm l_\beta)=0,\]
for some constant $l_\beta>0$, corresponds to the trajectory $p^2=2\int_q^\beta f(v)dv$, which is denoted by $\Gamma_\beta$.
Using the phase plane analysis, together with the condition {\bf{(H)}}, we are able to obtain the following lemma.

\begin{lem}\label{lem:ppa}
For any $\beta\in(0,\theta^*)$, there are exactly two points of intersection of $\Gamma_0$ and $\Gamma_\beta$. If $\beta=\theta^*$, there is a unique point $(\theta^*,0)$ of intersection of $\Gamma_0$ and  $\Gamma_{\theta^*}$.
If $\beta\in(\theta^*,1)$, then $\Gamma_0$ does not intersect $\Gamma_\beta$.
\end{lem}

It will turn out that bounded and nonnegative stationary solutions of \eqref{p} play a fundamental role in describing the long-time dynamics of system \eqref{p}. Based on Lemma \ref{lem:ppa}, we shall list all possible  bounded and nonnegative stationary solutions of \eqref{p} in the following lemma; one can also refer to Figure \ref{fig:1} for the structure of ground state.
\begin{figure}[htbp]
\centering
\includegraphics[width=7cm]{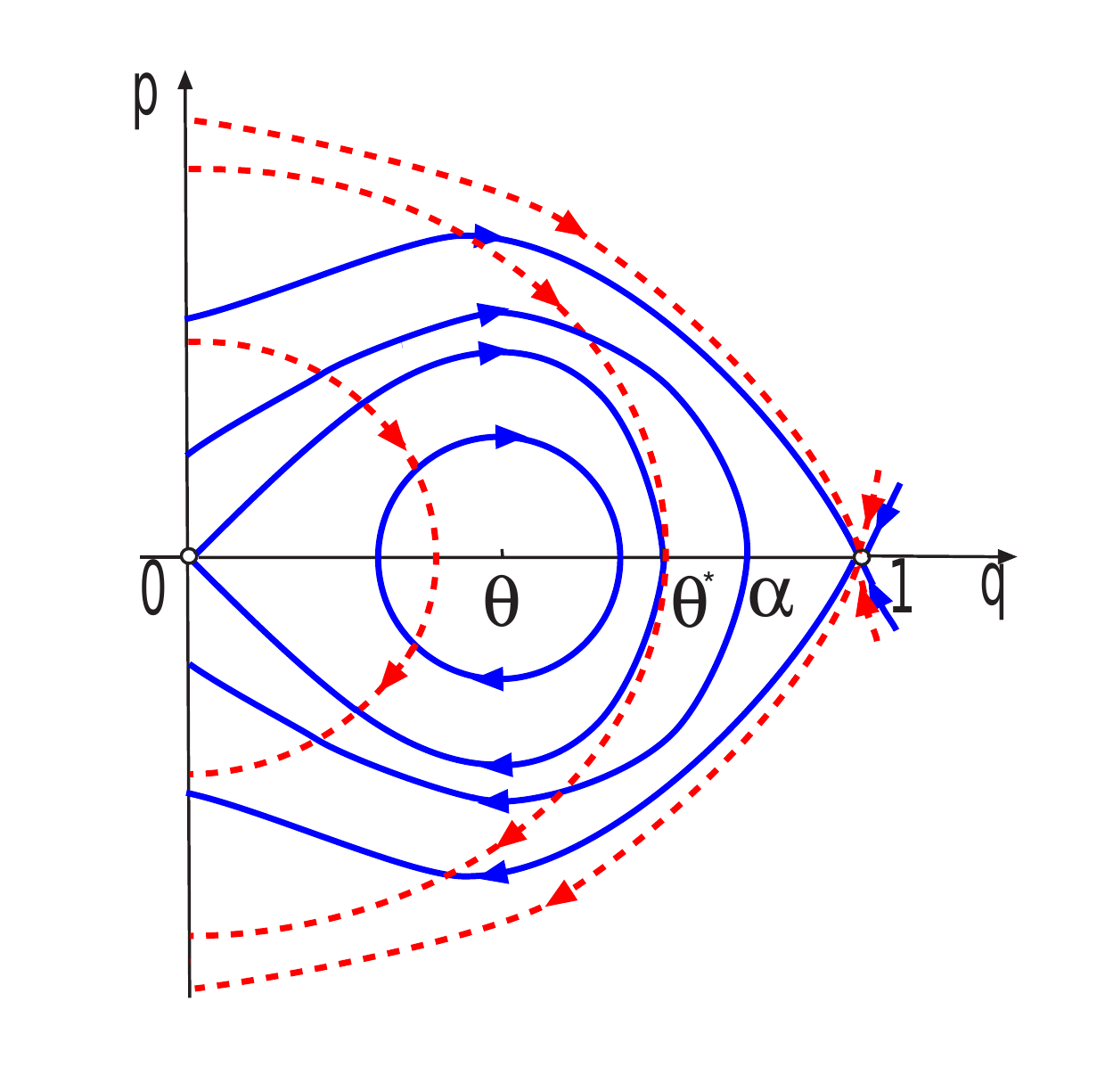}
\includegraphics[width=7cm]{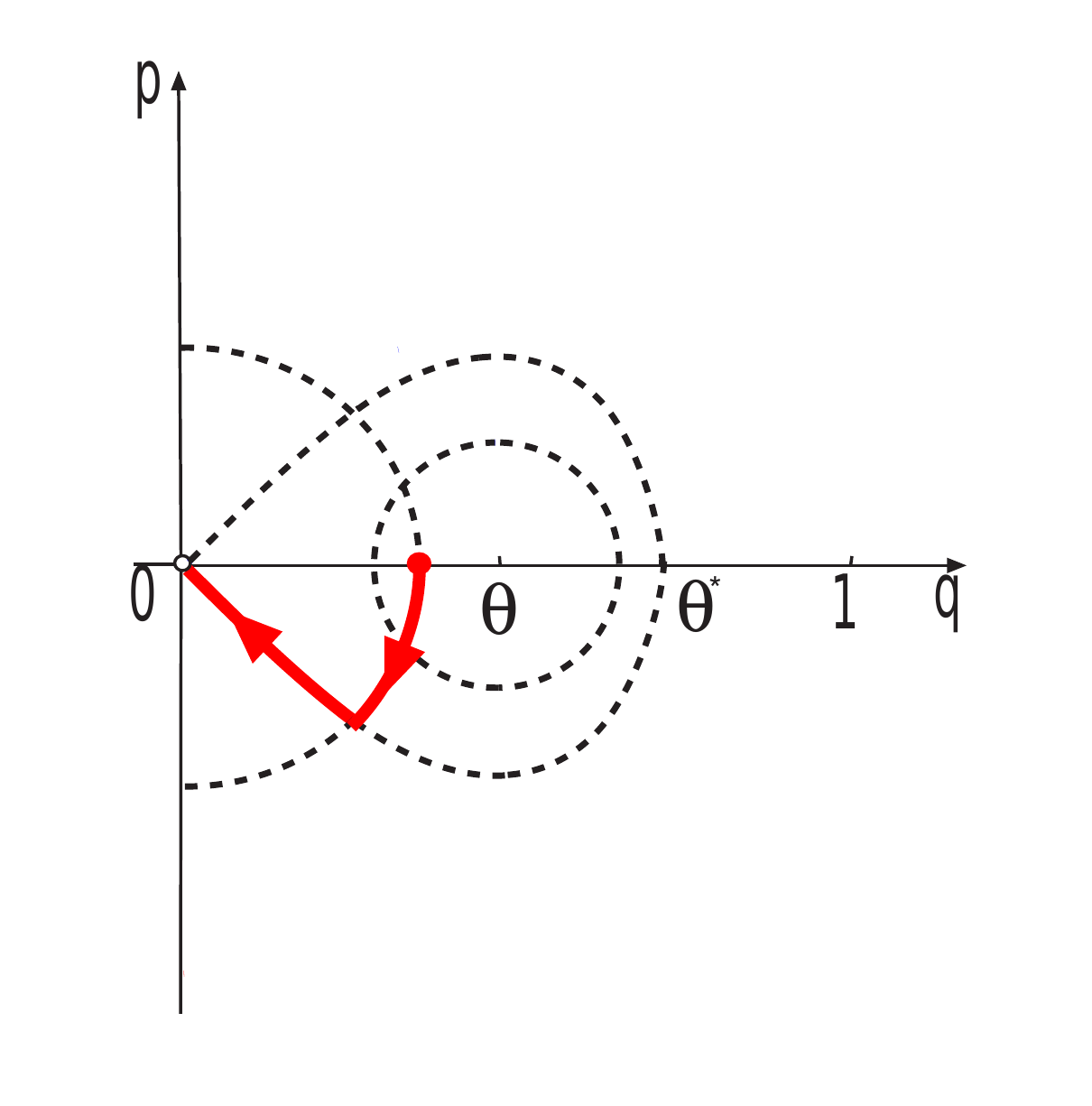}
\caption{{\small The $qp$-plane. Left: the red dotted curves are the trajectories for $q''+f(q)=0$, the blue solid curves are the trajectories for $q''+g(q)=0$; Right: the red solid curve is the trajectory for a ground state of \eqref{U}.} }\label{fig:1}
\end{figure}
\begin{lem}\label{lem:staso}
For any $L>0$, all solutions of the stationary problem \eqref{U} are one of the following types:
\begin{itemize}
\item[\vspace{10pt}(1)] \;Trivial solution: $U\equiv 0$;
\item[\qquad(2)] \;Positive constant solution: $U\equiv 1$;
\item[\qquad(3)] \;Ground state: $U$ is decreasing, $U(x)>0,\,\forall x\in[0,\infty)$, and
\[\left\{
\begin{array}{ll}
 U'' + f(U)=0, & 0<x<L,\\
 U'' + g(U)=0, &  x>L,\\
 U'(0)= 0, \\
 U(L-0)= U(L+0), \\
 U'( L-0)= U'(L+0).
\end{array}
\right.
\]
Moreover, when $x>L$, $U(\cdot)=V(\cdot-z)$ where $z\in\R$ and $V$ is the unique positive symmetrically decreasing solution of
 \[V'' + g(V)=0,\ \ \ V(0)=\theta^*,\ \ \ V(\pm\infty)=0.\]
\item[\qquad(4)] \;Positive periodic solution: $U(x)>0$ for $x\geq 0$ and when $x>L$, $U(x)=P(x-z_0)$, where $z_0\in\R$ and $P$ is a periodic solution of $P'' + g(P)=0$ satisfying  $0<\min P<\theta <\max P<\theta^*$.

\end{itemize}
\end{lem}

In addition, if the stationary solution $U(x)$ is a ground state, then it is easily seen from Lemma \ref{lem:ppa} that
\begin{equation}\label{Utot}
\|U\|_{L^\infty(0,\infty)}=U(0)< \theta^*.
\end{equation}

We also observe that, for any $\alpha\in(\theta^*,1)$, the trajectory for $q''+g(q)=0$ passing through the point $(\alpha,0)$ in the phase plane gives a function $v_\alpha$ satisfying
\begin{equation}\label{vala}
v_\alpha''+g(v_\alpha)=0<v_\alpha\leq \alpha\ \ \mbox{ in } (0,2l_\alpha),
\ \ \  v_\alpha(0)=v_\alpha(2l_\alpha)=0,\ \ \ v_\alpha(l_\alpha)=\alpha.
\end{equation}
where
\begin{equation}\label{la1}
l_\alpha=\int_0^\alpha\frac{ds}{\sqrt{G(s)-G(\alpha)}}\in(0,\infty).
\end{equation}

\vskip10pt

As far as the nonnegative stationary problem associated with \eqref{q} is concerned,
in addition to the constants $0$ and $1$, a ground state $U$ is a positive solution of the following elliptic problem:
\begin{equation}\label{q-statio}
\left\{
\begin{array}{ll}
 U'' + f(U)=0, &   x \in (L_1, L_2),\\
 U'' + g(U)=0, &  x\in(0, L_1)\cup(L_2,\infty),\\
 U'(0)=0, & \\
 U( L_i-0)= U( L_i+0), & i=1, 2, \\
 U'( L_i-0)= U'( L_i+0), & i=1, 2.
\end{array}
\right.
\end{equation}
Moreover, when $x>L_2$, $U(\cdot)=V(\cdot-z_1)$, where $z_1\in\R$ and $V$ is the unique positive decreasing solution of
 \[V'' + g(V)=0,\ \ \ V(0)=\theta^*,\ \ \ V(\pm\infty)=0.\]
 When $x\in[0,L_1]$, $U(\cdot)=P(\cdot-z_2)$ where $z_2\in\R$ and $P$ is a periodic solution of
$P'' + g(P)=0$ satisfying  $0<\min P<\theta <\max P<\theta^*$.

It is observed that \eqref{q-statio} may have eight types of ground state; we refer the reader to Figure \ref{fig:2}.
Moreover, different from \eqref{U}, a ground state of \eqref{q-statio} may not be decreasing with respect to $x\in[0,\infty)$. In spite of this, it can be shown that any ground state $U$ of \eqref{q-statio} satisfies that
 \begin{equation}\label{Utot-a}
\|U\|_{L^\infty(0,\infty)}< \theta^*.
\end{equation}

\begin{figure}[htbp]
\centering
\includegraphics[width=3.9cm]{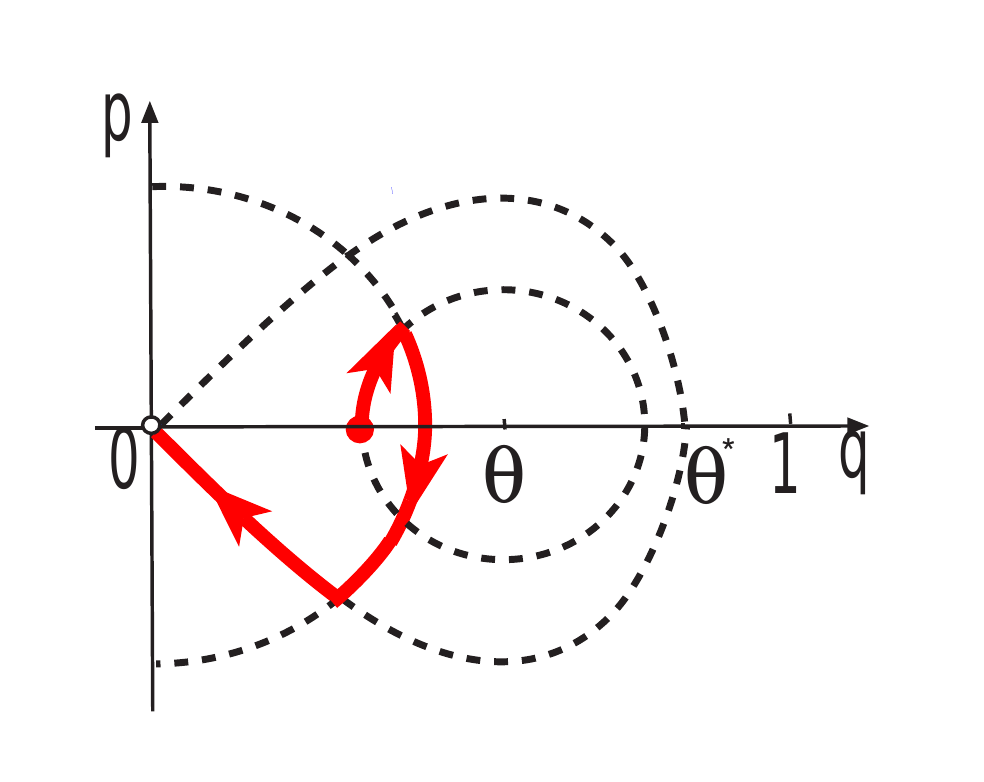}
\includegraphics[width=3.9cm]{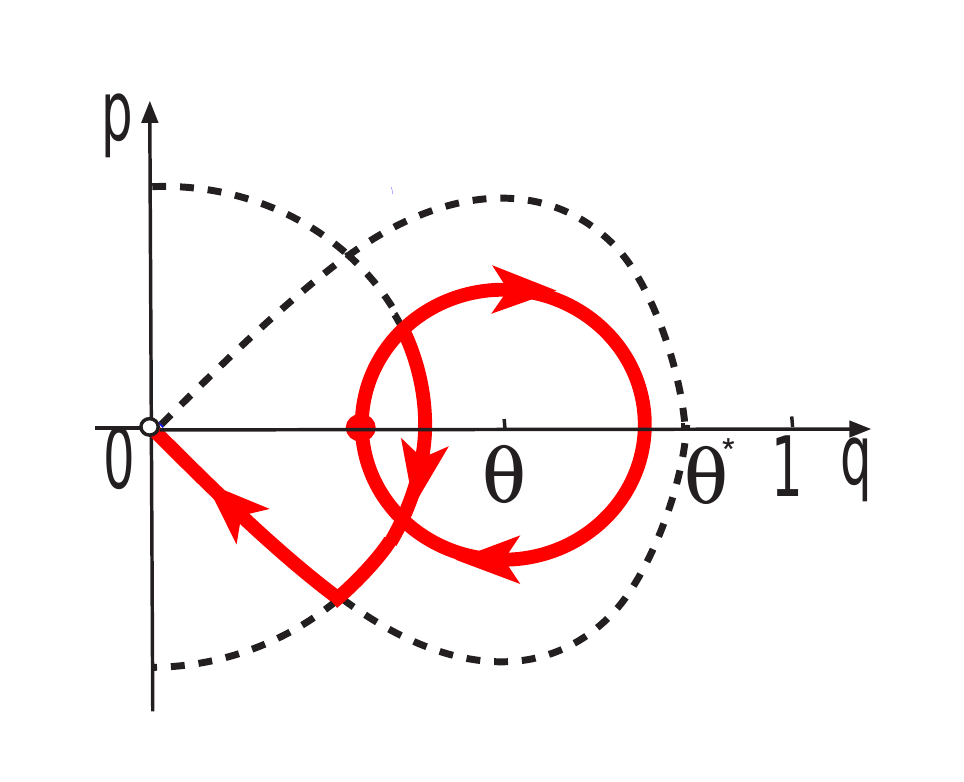}
\includegraphics[width=3.9cm]{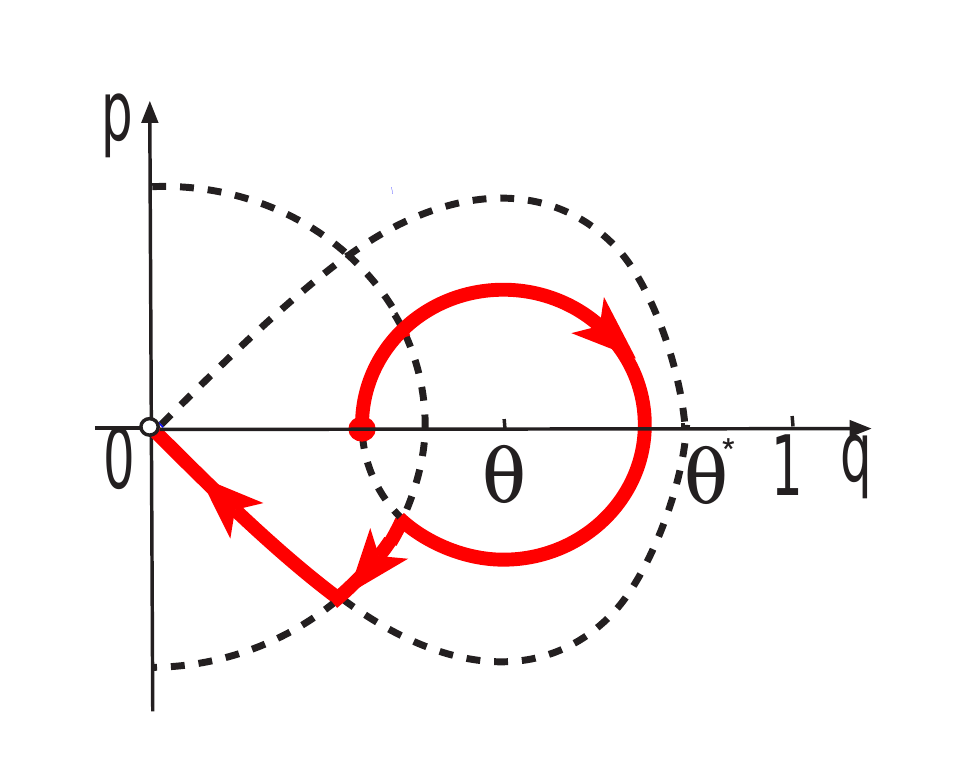}
\includegraphics[width=3.9cm]{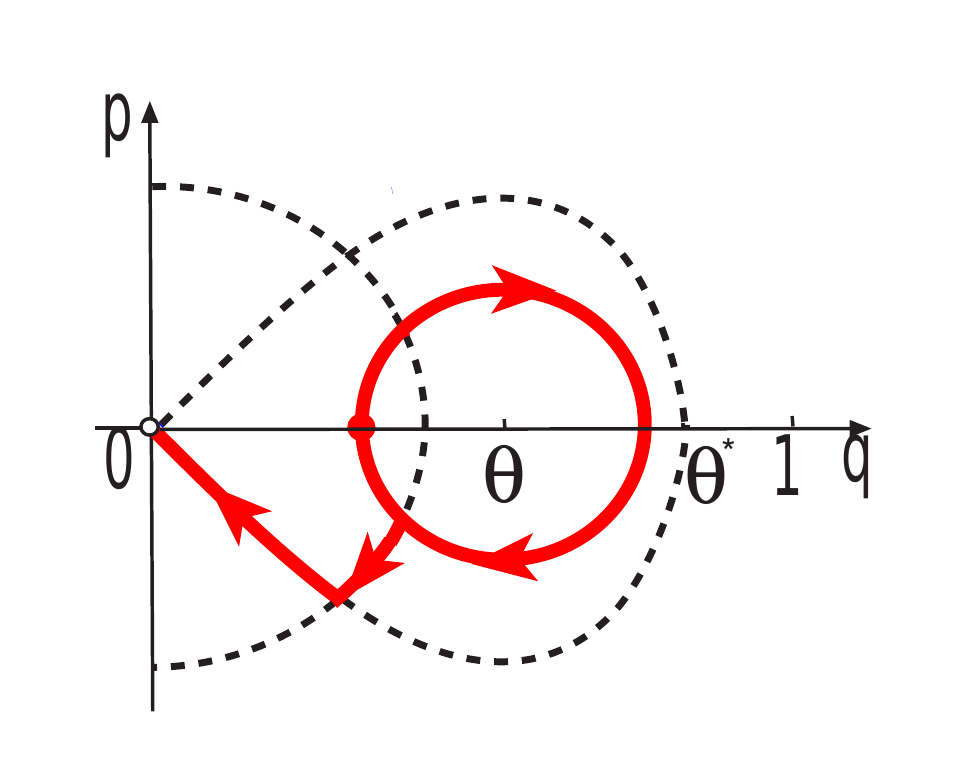}
\includegraphics[width=3.9cm]{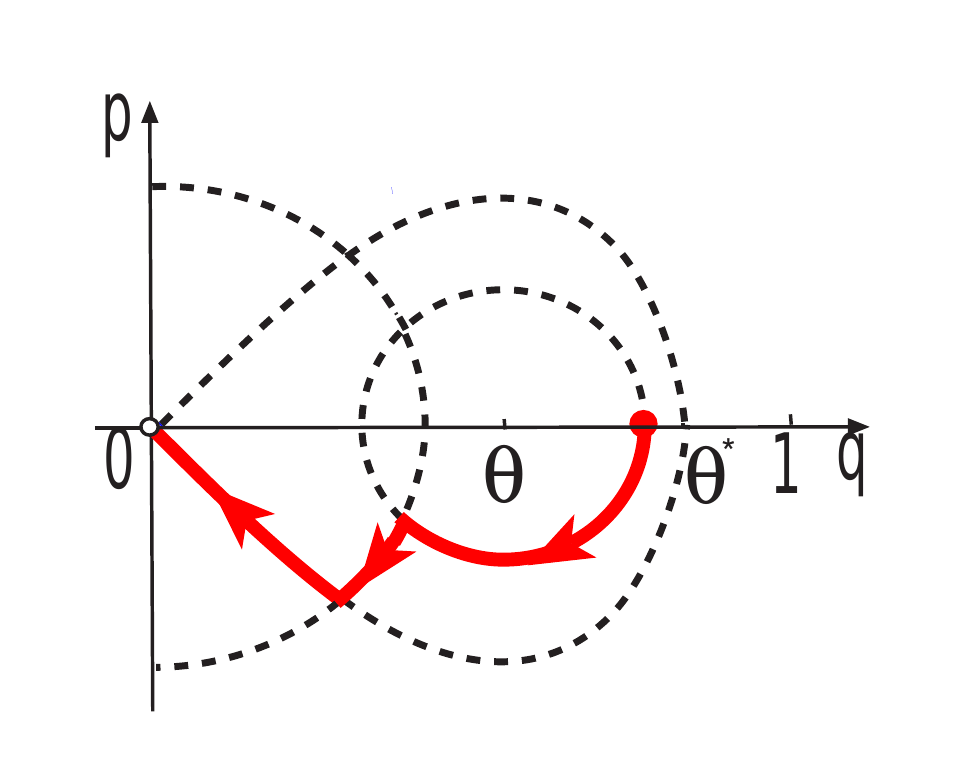}
\includegraphics[width=3.9cm]{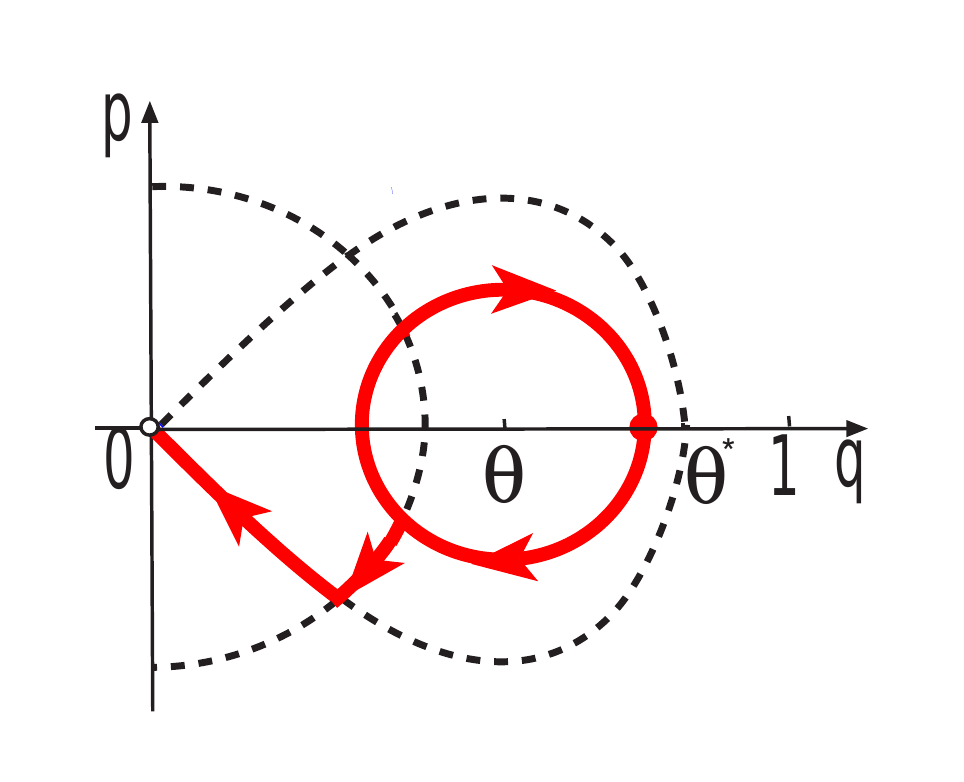}
\includegraphics[width=3.9cm]{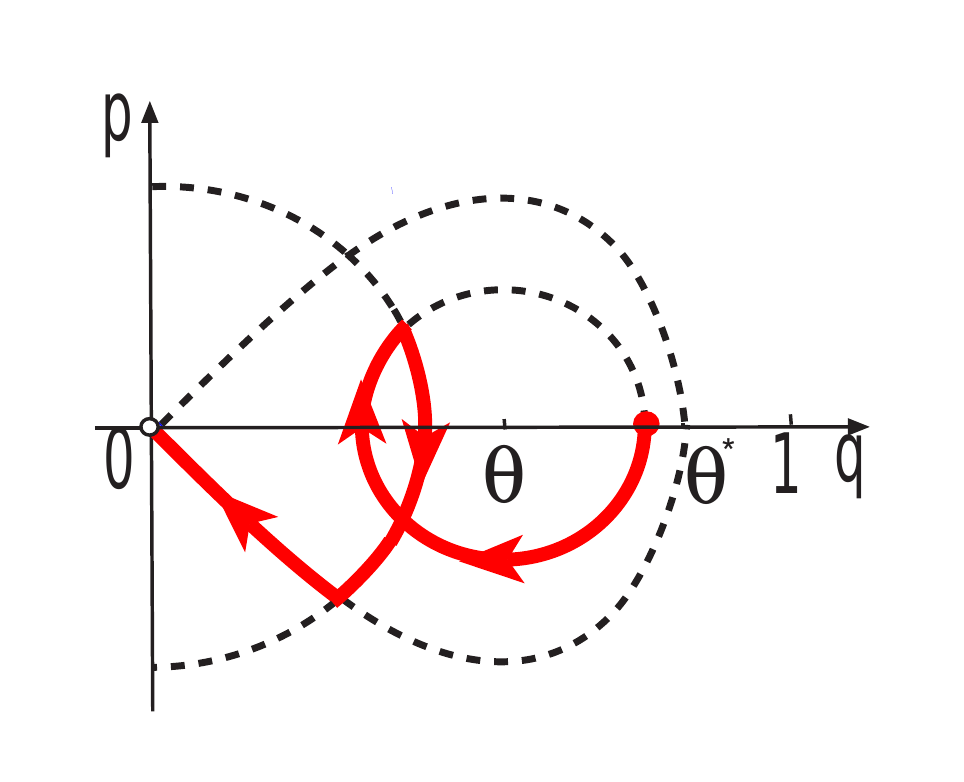}
\includegraphics[width=3.9cm]{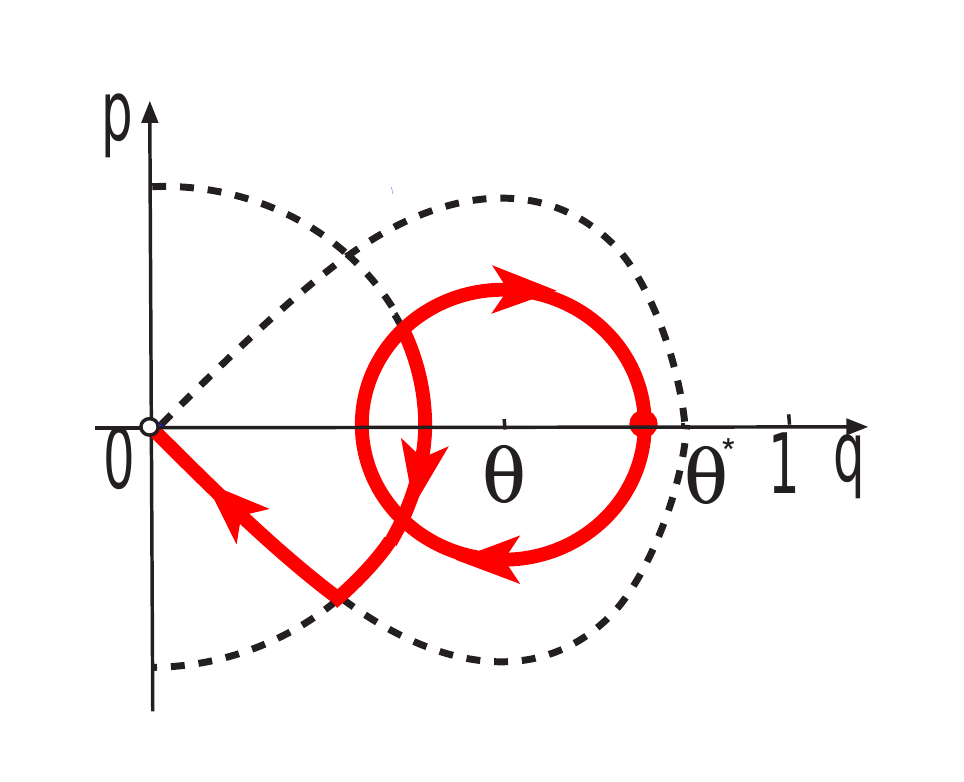}
\caption{{\small Eight types of ground state for system \eqref{q-statio}. } }\label{fig:2}
\end{figure}

\subsection{A general convergence theorem}\label{sub:zeronumber}

By the similar analysis as in \cite{DM}, we can present a general convergence result, which reads as follows.

\begin{thm}\label{thm:convergence}{\rm(Convergence theorem for systems \eqref{p} and \eqref{q})}
Let $u(t,x)$ be a solution of \eqref{p} (or \eqref{q}) with $u_0\in X^+$. Then $u$ converges to a stationary
solution $U$ as $t\to \infty$ locally uniformly in $[0,\infty)$. Moreover, the limit $U$ is either a constant solution or a
ground state of \eqref{U} (or \eqref{q-statio}).
\end{thm}

\begin{proof}\ We only sketch the proof for system \eqref{p}; the analysis for system \eqref{q} is similar.

Denote by $\omega(u)$ the $\omega$-limit set of $u(t,\cdot)$ in the topology
 of $L^{\infty}_{loc}([0,\infty))$. By local parabolic estimates, the definition of $\omega(u)$
 remains unchanged if the topology of  $L^{\infty}_{loc}([0,\infty))$ is replaced by that of $C^2_{loc}([0,L)\cup(L,\infty))\cap C^1_{loc}([0,\infty))$. It is well-known that $\omega(u)$ is a compact, connected and invariant set.

Following the argument of \cite[Theorem 1.1]{DM} with slight modifications, we can show that
$\omega(u)$ consists of only one element, which is either a constant solution or a decreasing solution of \eqref{U}.
In view of Lemma \ref{lem:staso}, $\omega(u)$ contains either $0$, or $1$, or a ground state of \eqref{U}.
We omit all the details but want to point out one thing in the proof: the condition in \eqref{p} at the connection point $x=L$ does not change the number of sign changes of the functions defined similarly as in \cite[Lemmas 2.7, 2.9]{DM}, due to the classical strong maximum principle and Hopf boundary lemma for parabolic equations.

\end{proof}

\section{A Connected Protection Zone: Proof of Theorem \ref{thm:dybe}}\label{sec:spanv}

This section is devoted to the proof of Theorem \ref{thm:dybe}. To achieve the aim,
we need to establish a series of lemmas in the coming three subsections.

\subsection{An eigenvalue problem}\label{subs:va} We first study the following eigenvalue problem:
 \begin{equation}\label{eigen-p}
 \left\{
 \begin{array}{ll}
 - \varphi'' -f'(0)\varphi =\lambda \varphi, & 0<x<L,\\
 - \varphi'' -g'(0)\varphi =\lambda \varphi, & x>L,\\
 \varphi'(0) = \varphi(\infty)=0, &  \\
 \varphi(L-0) = \varphi( L+0),\\
 \varphi'(L-0) = \varphi'( L+0),
 \end{array}
 \right.
 \end{equation}
and analyze the properties of its principal eigenvalue. It will turn out that these preliminary results are crucial
in determining the dynamics of \eqref{p}.

For our purpose, let us consider the following eigenvalue problem on $\R$ associated with \eqref{eigen-p}:
 \begin{equation}\label{eigen-p1}
 \left\{
 \begin{array}{ll}
 - \varphi'' +h(x)\varphi =\lambda \varphi, & -\infty<x<\infty,\\
 \varphi(-\infty) = \varphi(\infty)=0,
 \end{array}
 \right.
 \end{equation}
where
 $$
h(x) = \left\{
\begin{array}{ll}
     -f'(0),\ \ & x\in[-L,L],\\
      -g'(0), \ \ & x\in\R\setminus{[-L,L]}.
\end{array}
\right.
$$

As $h\in L^\infty(\R)$ and is symmetric with respect to the origin,
it is well known that the principal eigenvalue (or the so-called first eigenvalue) of \eqref{eigen-p1} exists and
coincides with that of problem \eqref{eigen-p}. Thus, we use $\lambda_1 (L)$ to denote the principal eigenvalue of
\eqref{eigen-p} and \eqref{eigen-p1}. The corresponding eigenfunction $\varphi_1^L$ of \eqref{eigen-p} satisfies
$\varphi_1^L\in C^1([0,\infty))\cap(C^2([0,\infty)\setminus\{L\}))$, $\varphi_1^L>0$ on $[0,\infty)$ and $(\varphi_1^L)'(0)=0$.

Let $\lambda_1^R(L)$ be the principal eigenvalue of
 \begin{equation}\label{eigen-p2}
 \left\{
 \begin{array}{ll}
 - \varphi'' +h(x)\varphi =\lambda \varphi, & -R<x<R,\\
 \varphi(-R) = \varphi(R)=0.
 \end{array}
 \right.
 \end{equation}
From \cite[Proposition 6.11]{BHR} (or \cite[Theorem 4.1]{BR}) it follows that
 \begin{equation}\label{eig-1}
 \mbox{$\lambda_1^R(L)$\ is decreasing in\ $R>0$}\ \, \mbox{and}\ \,\lim_{R\to\infty}\lambda_1^R(L)\leq\lambda_1 (L).
 \end{equation}

Let $L_*$ be given as in Section 1. Then we are able to conclude that

\begin{lem}\label{lem:1eigenvalue} Let $\lambda_1 (L)$ be the principal eigenvalue of \eqref{eigen-p}. Then
\[
\lambda_1 (L)\in(-f'(0),-g'(0)),\ \ \mbox{ for any}\ L>0,
\]
and
\[
L=\frac{1}{\sqrt{f'(0)+\lambda_1(L)}}\arctan\sqrt{-\frac{g'(0)+\lambda_1(L)}{f'(0)+\lambda_1(L)}}.
\]
Moreover, $\lambda_1(L)$ is decreasing with respect to $L>0$, and $\lambda_1(L)<0$ if $L>L_*$, $\lambda_1(L)=0$ if $L=L_*$, and $\lambda_1(L)>0$ if $0<L<L_*$.

\end{lem}

\begin{proof} For simplicity, we write $\lambda_1=\lambda_1(L)$, and $\varphi(x)$ is denoted to be a corresponding eigenfunction.

First of all, we claim that $\lambda_1\in(-f'(0),-g'(0))$ for any $0<L<\infty$. We prove this claim in four steps as follows.

{\it Step 1: $\lambda_1\not=-g'(0)$.} If there is $L_0\in(0,\infty)$ such that $\lambda_1=-g'(0)$,
then for $x>L_0$, it follows from the second equation of \eqref{eigen-p} that
\[
-\varphi''=(g'(0)+\lambda_1)\varphi=0,
\]
which implies that there is a constant $C$ such that $\varphi'(x)\equiv C$ for $x>L_0$.
This, together with $\varphi(\infty)=0$, shows that $\varphi(x)\equiv 0$, a contradiction.

{\it Step 2: $\lambda_1<-g'(0)$.} Suppose that $\lambda_1>-g'(0)$ for $L_0\in(0,\infty)$.
For $x> L_0$, it holds
\[
-\varphi''=(g'(0)+\lambda_1)\varphi\ \ \mbox{and}\ \ g'(0)+\lambda_1>0.
\]
This, together with $\varphi(x)>0$ for $x\geq0$, yields that
$\varphi''<0$ in $(L_0,\infty)$. Therefore, $\varphi'(x)$ is decreasing in $x>L_0$. If $\varphi'(\infty)\geq 0$,
then $\varphi$ is increasing in $x> L_0$. As $\varphi>0$ on $[L_0,\infty)$, we arrive at a contradiction with $\varphi(\infty)=0$.
If $\varphi'(\infty)<0$, then it is easily shown that there exists a large $x_0>0$
such that $\varphi(x)<0$ for $x\in[x_0,\infty)$, which again leads to a contradiction.

{\it Step 3: $\lambda_1\not=-f'(0)$.} If there is $L_0\in(0,\infty)$ such that $\lambda_1=-f'(0)$. Then for $0\leq x\leq L_0$, it follows from the first
equation of \eqref{eigen-p} that
\[
-\varphi''=(f'(0)+\lambda_1)\varphi=0.
\]
This, together with $\varphi'(0)=0$, implies that
\begin{equation}\label{ee0}
\varphi'(x)\equiv 0,\ \ \ \forall x\in[0,L_0].
\end{equation}

For $x> L_0$, it follows from the second equation of \eqref{eigen-p} that
\[
-\varphi''=(g'(0)+\lambda_1)\varphi=(g'(0)-f'(0))\varphi,
\]
which is a second-order linear ordinary differential equation.
Since $g'(0)-f'(0)<0$, one can find two constants $C_1$ and $C_2$ such that
\[
\varphi(x)=C_1e^{\sqrt{f'(0)-g'(0)}\,x}+C_2e^{-\sqrt{f'(0)-g'(0)}\,x} \ \ \mbox{ for } x>L_0.
\]
Note that $\varphi(\infty)=0$. We have $C_2>0=C_1$. In turn,
$\varphi(x)=C_2e^{-\sqrt{f'(0)-g'(0)}\,x}$ for
$x>L_0$. Hence, we obtain
\[
\varphi'(L_0+0)=-C_2\sqrt{f'(0)-g'(0)}e^{-\sqrt{f'(0)-g'(0)}\,L_0}<0=\varphi'(L_0-0),
\]
due to \eqref{ee0}. This leads to a contradiction with the condition $\varphi'(L_0-0)=\varphi'(L_0+0)$.

{\it Step 4: $\lambda_1>-f'(0)$.} If there is $L_0\in(0,\infty)$ such that $\lambda_1<-f'(0)$. Then for $0\leq x\leq L_0$, it follows from the first
equation of \eqref{eigen-p} that
\[
-\varphi''=(f'(0)+\lambda_1)\varphi\ \ \mbox{with}\ f'(0)+\lambda_1<0.
\]
Thus there are two constants $C_3$ and $C_4$ such that
\[
\varphi(x)=C_3e^{\sqrt{-(f'(0)+\lambda_1)}\,x}+C_4e^{-\sqrt{-(f'(0)+\lambda_1)}\,x} \ \ \mbox{ for } x\in[0,L_0].
\]
It is easy to compute that
\[
\varphi'(0)=\sqrt{-(f'(0)+\lambda_1)}(C_3-C_4).
\]
In view of $\varphi'(0)=0$, $C_3=C_4>0$, and so
\[
\varphi(x)=C_3\big[e^{\sqrt{-(f'(0)+\lambda_1)}\,x}+e^{-\sqrt{-(f'(0)+\lambda_1)}\,x}\big] \ \ \mbox{ for } x\in[0,L_0].
\]
This infers that
\begin{equation}\label{dde0}
\varphi'(L_0-0)=C_3\sqrt{-(f'(0)+\lambda_1)}\big(e^{\sqrt{-(f'(0)+\lambda_1)}L_0}-e^{-\sqrt{-(f'(0)+\lambda_1)}L_0}\big)>0.
\end{equation}

On the other hand, when $x> L_0$, it follows from the second equation of \eqref{eigen-p} and the claims proved in Steps 1 and 2 that
\[
-\varphi''=(g'(0)+\lambda_1)\varphi\ \ \mbox{with}\ g'(0)+\lambda_1<0.
\]
Then we can find two constants $C_5$ and $C_6$  such that
\[
\varphi(x)=C_5e^{\sqrt{-(g'(0)+\lambda_1)}\,x}+C_6e^{-\sqrt{-(g'(0)+\lambda_1)}\,x} \ \ \mbox{ for } x>L_0.
\]
As $\varphi(\infty)=0$, it is necessary that $C_6>0=C_5$. Hence, $\varphi(x)=C_6e^{-\sqrt{-(g'(0)+\lambda_1)}x}$ for
$x>L_0$ and
\[
\varphi'(L_0+0)=-C_6\sqrt{-(g'(0)+\lambda_1)}e^{-\sqrt{-(g'(0)+\lambda_1)}L_0}<0.
\]
Using this, \eqref{dde0} and the condition $\varphi'(L_0-0)=\varphi'(L_0+0)$, we arrive at a contradiction.

A combination of the results proved in the above four steps establishes our previous claim.

Now, for $x\in[0,L]$, we obtain from the first equation of \eqref{eigen-p} that
\[
-\varphi''=(f'(0)+\lambda_1)\varphi\ \ \mbox{and}\ \ f'(0)+ \lambda_1>0.
\]
Then there exist two constants $C_7$ and $C_8$  such that
\[
\varphi(x)=C_7\cos \sqrt{f'(0)+ \lambda_1}\,x +C_8\sin \sqrt{f'(0)+ \lambda_1}\,x \ \ \mbox{ for } x\in[0,L].
\]
The facts $\varphi'(0)=0$ and $\varphi>0$ for $x\geq 0$ yield $C_7>0=C_8$. Thus,
\[
\varphi(x)=C_7\cos \sqrt{f'(0)+ \lambda_1}\,x \ \ \mbox{ for } x\in[0,L],
\]
and
\[
\frac{\varphi'(L-0)}{\varphi(L-0)}=-\sqrt{f'(0)+\lambda_1}\tan{\sqrt{f'(0)+\lambda_1}L},
\]
with $0<\sqrt{f'(0)+\lambda_1}L<\frac{\pi}{2}$.

Similarly, it follows from the second equation of \eqref{eigen-p} that
\[
-\varphi''=(g'(0)+\lambda_1)\varphi\ \ \mbox{and}\ \ g'(0)+ \lambda_1<0.
\]
Clearly, it holds
\[
\varphi(x)=C_9e^{-\sqrt{-(g'(0)+\lambda_1)}\,x}+C_{10}e^{\sqrt{-(g'(0)+\lambda_1)}\,x} \ \ \mbox{ for } x>L,
\]
for constants $C_9,\,C_{10}$.
Since $\varphi(\infty)=0$, then $C_9>0=C_{10}$, and
\[
\varphi(x)=C_9e^{-\sqrt{-(g'(0)+\lambda_1)}x} \ \ \mbox{ for } x>L,
\]
and
\[
\frac{\varphi'(L+0)}{\varphi(L+0)}=-\sqrt{-g'(0)-\lambda_1}.
\]
Thus we have
\[
\sqrt{f'(0)+\lambda_1}\tan{\sqrt{f'(0)+\lambda_1}L}=\sqrt{-g'(0)-\lambda_1},
\]
that is,
\[
L=\frac{1}{\sqrt{f'(0)+\lambda_1}}\arctan\sqrt{-\frac{g'(0)+\lambda_1}{f'(0)+\lambda_1}}\ \ \mbox{ and }\ 0<\sqrt{f'(0)+\lambda_1}L<\frac{\pi}{2}.
\]

It is obvious that $\lambda_1$ is decreasing in $L>0$. Moreover, it is easily checked that when
$L=L_*:=\frac{1}{\sqrt{f'(0)}}\arctan\sqrt{-\frac{g'(0)}{f'(0)}}$, then $\lambda_1=0$, and
\[
\lim_{L\to0}\lambda_1=-g'(0)>0,\ \ \ \ \ \ \ \lim_{L\to\infty}\lambda_1=-f'(0)<0.
\]
Thanks to the monotonicity of $\lambda_1$ in $L$, we can derive all the other assertions of the lemma.
\end{proof}

\subsection{Sufficient conditions for vanishing}
Now we give some sufficient conditions for vanishing. Let $u_\sigma$ be a solution of \eqref{p} with $u_0=\phi_\sigma\in X^+$, and define
 $$
 \Sigma_0=\{\sigma>0:\ u_\sigma(t,x)\to0\ \mbox{uniformly on} \ \R_+,\ \mbox{as}\ t\to\infty\}.
 $$
Then we have

\begin{lem}\label{lem:vanishing}
Let $L_*$ be given in Lemma \ref{lem:1eigenvalue} and assume that $(\Phi_1)$--$(\Phi_3)$ hold.
If $0<L<L_*$, then $\Sigma_0$ is nonempty.
\end{lem}

\begin{proof} By Lemma \ref{lem:1eigenvalue}, the eigenvalue problem \eqref{eigen-p} with $L\in(0,L_*)$ admits a positive principal
eigenvalue $\lambda_1:=\lambda_1(L)$, and the corresponding positive eigenfunction $\varphi_1^L$ can be normalized so that $\|\varphi_1^L\|_{L^{\infty}}=1$. Set
\[
w(t,x) :=\delta e^{-\frac{\lambda_1}{2}t}\varphi_1^L(x)\ \mbox{ for } t\geq0, x \geq 0,
\]
with some $\delta>0$ so small that
\begin{equation}\label{delt}
f(s)\leq\big(f'(0)+\frac{\lambda_1}{2}\big)s\ \mbox{ and }\ g(s)\leq\big(g'(0)+\frac{\lambda_1}{2}\big)s\ \ \mbox{ for } s\in[0,\delta].
\end{equation}
A simple calculation yields that for $t>0$ and $0\leq x \leq L$,
$$
w_t-w_{xx}-f(w)  = w\big(\frac{\lambda_1}{2}+f'(0)\big)-f(w) \geq 0,
$$
and for $t>0$ and $x> L$,
$$
w_t-w_{xx}-g(w)= w\big(\frac{\lambda_1}{2}+g'(0)\big)-g(w) \geq 0,
$$
where we have used \eqref{delt}. Since $\varphi$ is the principal eigenfunction, clearly
$w(t,L-0) = w(t, L+0)$ and  $w'(t,L-0) = w'(t, L+0),\,\forall t>0$.

Therefore, $w$ will be a supersolution of \eqref{p} if $u_0(x)=\phi_\sigma(x)\leq w(0,x)$ in $[0,\hbar]$, where $\hbar$ is defined in Section 1. Choose $\hat{\sigma}:=\delta\min_{x\in[0,\hbar]}\varphi_1^L(x)$,
then when $\|u_0\|_{L^\infty}\leq \hat{\sigma}$ we have $u_0(x)\leq w(0,x)$ in $[0,\hbar]$.
Thus, $w$ is a supersolution of \eqref{p}. The comparison principle can be used to deduce that $u_\sigma(t,x)\leq w(t,x)$ for $t\geq0$ and $x \geq 0$. This, together with
the fact that $\|w\|_{L^\infty([0,\infty))}\to 0$ as $t\to \infty$, implies that vanishing happens for $u_\sigma$. This proves the lemma.
\end{proof}

From Lemma \ref{lem:vanishing} and its proof, we immediately obtain

\begin{cor}\label{cor:BC-vanishing-finite}
Let $\delta$ and $\varphi_1^L$ be given in the proof of Lemma \ref{lem:vanishing}. If $0<L<L_*$ and there is some $t_0>0$ such that
\[
u_\sigma(t_0,x)\leq \delta \varphi_1^L(x),\ \ \mbox{ for all}\ x\in[0,\infty),
\]
then vanishing happens for $u_\sigma$.
\end{cor}

We next establish two decay properties of solutions at infinity w.r.t $x$.
\begin{lem}\label{lem:decay}
 Let $0<L<L_*$ and  $u_\sigma$ be a solution of \eqref{p} with $u_0=\phi_\sigma\in X^+$.  Assume that vanishing
 happens for $u_\sigma$. Then there are constants $\kappa_1>0$ and $x_1\gg 1$ such that
 \begin{equation}\label{dequ}
 u_\sigma(t,x)\leq \kappa_1 e^{-\sqrt{-g'(0)-\epsilon_0}\,x},
 \ \ \ \ \forall t\geq0,\ x\geq x_1,
 \end{equation}
where $\epsilon_0=\frac{1}{2}\min\{-g'(0),\ \lambda_1(L)\}>0$.
\end{lem}
\begin{proof} It is well known (see, e.g., \cite{DuLou,DM}) that the following problem
\begin{equation}\nonumber
v''+g(v)=0\ \ \ \mbox{in }\, \R
\end{equation}
admits a unique positive symmetrically decreasing solution $\Phi(x)$, which satisfies
\begin{equation}\nonumber
\Phi(0)=\theta^*,\ \ \ \Phi'(0)=0,\ \ \ \lim_{|x|\to\infty}\Phi(x)=0.
\end{equation}

Recall that $\lambda_1(L)>0$ if $0<L<L_*$. Thus, one can find a large $x_1>0$ so that
 \begin{equation}\nonumber
 \left\{
 \begin{array}{ll}
 -\Phi''=g(\Phi)\leq \big(g'(0)+\epsilon_0\big)\Phi, & x>x_1,\\
 \Phi(\infty)=0.
 \end{array}
 \right.
 \end{equation}
As a consequence, it is easily seen that
 \begin{equation}\label{vell-a}
 \Phi(x)\leq \kappa_1e^{-\sqrt{-g'(0)-\epsilon_0}\,x},\ \ \ x>x_1,
 \end{equation}
for some constant $\kappa_1>0$.

On the other hand, because vanishing happens for $u_\sigma$, it then follows from \cite{P}
(refer to Lemmas 4.3-4.6 there) that there is $\rho>0$ such that
\[
u_\sigma(t,x)<\theta^*=\Phi(x-\rho)\ \ \mbox{ for } t\geq 0,\ x=\rho.
\]
Thus the comparison principle gives that $u_\sigma(t,x)\leq \Phi(x-\rho)$ for $t\geq 0$ and $x\geq \rho$.
This and \eqref{vell-a} imply the estimate in \eqref{dequ} for $u_\sigma$ immediately.
\end{proof}

 \begin{lem}\label{lem:decay1}  Let $u_\sigma$ be a solution of \eqref{p} with $u_0=\phi_\sigma\in X^+$.
 Then there are two positive constants $\kappa_2$ and $M$ such that
 \begin{equation}\label{dequ1}
 u_\sigma(t,x)\leq \kappa_2 e^{Mt-\frac{(x-\hbar)^2}{4t}},\ \ \ \ \forall t\geq1,\ \ x\geq \hbar.
 \end{equation}
\end{lem}
\begin{proof} Since $f,\ g$ are globally Lipschitz, we can find a positive constant $M$ such that
\[
f(u),\ g(u) \leq M u,\ \ \mbox{ for } u\geq0.
\]
Consider the Cauchy problem
\begin{equation}\label{prob-w}
\left\{
\begin{array}{ll}
 \bar{u}_t =\bar{u}_{xx} +M\bar{u}, & x\in\R,\ t>0,\\
 \bar{u}(0,x)= \bar{u}_0 (x), &  x \in \R,
  \end{array}
 \right.
 \end{equation}
where
$$
\bar{u}_0 (x) = \left\{
\begin{array}{ll}
     u_0(|x|),\ \ & x\in (-\hbar, \hbar),\\
      0, \ \ & x\not\in (-\hbar, \hbar).
\end{array}
\right.
$$
Then we have the explicit expression for $\bar{u}$:
$$
\bar{u} (t,x)  =  \frac{e^{Mt}}{\sqrt{4\pi t}} \int_{-\hbar}^\hbar
e^{-\frac{(x-y)^2}{4t}} {\bar{u}}_0 (y)dy  \leq \kappa_2 e^{Mt-\frac{(x-\hbar)^2}{4t}},\ \ \forall t\geq1, x\geq \hbar,
$$
for some constant $\kappa_2>0$.

By the standard comparison theorem, we conclude that $u_\sigma(t,x) \leq
\bar{u}(t,x)$ for $t\geq1$ and $x\geq\hbar$, and the desired
inequality follows.
\end{proof}

We end this subsection with some useful properties of the set $\Sigma_0$ for vanishing.
\begin{lem}\label{lem:sd0} Let $\Sigma_0$ be defined as before. The following assertions hold.
 \begin{itemize}
\item[\vspace{10pt}(i)] \;If $0<L<L_*$, then the set $\Sigma_0$ is an open interval $(0,\sigma_*)$
                      for some $\sigma_*\in(0,\infty];$
\item[\qquad(ii)] \;If $L>L_*$, then the set $\Sigma_0$ is empty.
\end{itemize}
\end{lem}

\begin{proof} (i) \ Since $0<L<L_*$, it follows from Lemma \ref{lem:vanishing} and the parabolic comparison principle that
vanishing happens for all small $\sigma>0$, thus $\Sigma_0$ is not empty in this case.

Next we want to show that $\Sigma_0$ is an open interval. Fix any $\sigma_0\in \Sigma_0$, then vanishing happens for $\sigma=\sigma_0$,
and so the solution $u_{\sigma_0}$ of \eqref{p} with $u_0=\phi_{\sigma_0}$ satisfies \eqref{dequ} for some $\kappa_1>0$ and $x_1\gg 1$.
On the other hand, from the proof of Lemma \ref{lem:1eigenvalue} it follows that
\[
\varphi_1^L(x)= C_2e^{-\sqrt{-g'(0)-\lambda_1(L)}\,x}, \ \ \forall x>L,
\]
for some constant $C_2>0$. Due to $\epsilon_0<\lambda_1(L)$, this and \eqref{dequ}
enable one to find a constant $x_2>x_1+L$ large enough so that
\[
u_{\sigma_0}(t,x)<\delta\varphi_1^L(x),\ \ \forall x\geq x_2, \ t\geq0,
\]
where $\delta>0$ is given in the proof of Lemma \ref{lem:vanishing}. In addition, as vanishing happens for $u_{\sigma_0}$, there is $T_0>1$ such that
\[
u_{\sigma_0}(T_0,x)< \delta\min_{x\in[0,x_2]}\varphi_1^L(x)\leq \delta\varphi_1^L(x),\ \ \forall 0\leq x \leq x_2.
\]
Consequently, we have
\[
u_{\sigma_0}(T_0,x)< \delta\varphi(x),\ \ \mbox{ for all}\ x\geq 0.
\]

By the continuous dependence of the solution of \eqref{p} on its initial values, combined with the estimate \eqref{dequ1}, we
can conclude that if $\epsilon>0$ is sufficiently small, the solution $u_{\sigma_0+\epsilon}$
of \eqref{p} with $u_0=\phi_{\sigma_0+\epsilon}$ satisfies
\[
u_{\sigma_0+\epsilon}(T_0,x)\leq \delta\varphi(x),\ \ \mbox{ for all} \ x\geq 0.
\]
Thanks to this, it follows from Corollary \ref{cor:BC-vanishing-finite} that vanishing happens for $u_{\sigma_0+\epsilon}$, which implies that $\sigma_0+\epsilon\in\Sigma_0$. Moreover, by the comparison principle, $\sigma\in \Sigma_0$ for any $\sigma<\sigma_0+\epsilon$.
Thus $\Sigma_0$ is an open interval. Define $\sigma_*:=\sup \Sigma_0$, then $\Sigma_0=(0, \sigma_*)$.

\smallskip

(ii) Let $u_\sigma$ be the solution of \eqref{p} with $u_0(x)=\phi_\sigma(x)$.
We first recall that $\lambda_1(L)<0$ if $L>L_*$. In light of \eqref{eig-1}, $\lambda_1^R(L)<0$ for all large $R$ with $R>L+\hbar$.
The positive eigenfunction corresponding to $\lambda_1^R(L)$, denoted by $\varphi_1^{L,R}$,  solves \eqref{eigen-p2} and can be normalized so that $\|\varphi_1^{L,R}\|_{L^{\infty}}=1$. Set
 $$
\underline u(x)= \left\{
\begin{array}{ll}
     \varrho\varphi_1^{L,R}(x),\ \ & x\in [0, R],\\
      0, \ \ & x\in (R, \infty),
\end{array}
\right.
$$
where the constant $\varrho>0$ can be chosen to be sufficiently small such that
\begin{equation}\nonumber
f(s)\geq\big(f'(0)+\frac{\lambda_1^R(L)}{2}\big)s\ \mbox{ and }\ g(s)\geq\big(g'(0)+\frac{\lambda_1^R(L)}{2}\big)s\ \ \mbox{ for } s\in[0,\varrho].
\end{equation}
As a result, for $t>1$ and $0\leq x \leq L$, it is easily seen that
$$
\underline u_t-\underline u_{xx}-f(\underline u) \leq \frac{\lambda_1^R(L)}{2}\underline u\leq 0,
$$
and for $t>1$ and $x> L$,
$$
\underline u_t-\underline u_{xx}-g(\underline u)\leq \frac{\lambda_1^R(L)}{2}\underline u\leq 0.
$$
By the definition of $\varphi_1^{L,R}(x)$ and $R>L+\hbar$, clearly
$\underline u(L-0) = \underline u(L+0)$ and  $\underline u'(L-0) = \underline u'(L+0)$.

Furthermore, since $u_\sigma(1,x)>0$ for all $x\geq0$, we can take $\varrho$ to be smaller if necessary such that
$u_\sigma(1,x)>\underline u(x)$ for all $x\geq0$. Hence, $\underline u$ is a generalized subsolution of \eqref{p} for $t\geq1,\,x\geq0$. By the comparison principle, we obtain $u_\sigma(t,x)\geq \underline u(x)$ for $t>1$ and $x \geq 0$.
This apparently implies that vanishing can not happen for $u_\sigma$, which
completes the proof of the lemma.
\end{proof}

\subsection{Sufficient conditions for spreading}\label{subs:spr}
Based on the phase plane analysis we can give the following sufficient condition for spreading.

\begin{lem}\label{lem:spreadc}
Let $u_\sigma$ be a solution of \eqref{p} with $u_0(x)=\phi_\sigma(x)\in X^+$.
If for each $\alpha\in(\theta^*,1]$, $u_0\geq \alpha$ on $[r,r+2l_\alpha]$ for some $r\geq L$, where $l_\alpha$ is given in \eqref{la1}, then spreading happens for $u_\sigma$.
\end{lem}
\begin{proof}
Let us consider the following auxiliary problem:
\begin{equation}\label{subso}
\left\{
\begin{array}{ll}
v_t = v_{xx} +g(v), & t> 0,\ x\in (r,\infty),\\
v(t,0)=0, & t>0,\\
v(0,x)=u_0(x), & x\in(r,\infty).
  \end{array}
 \right.
 \end{equation}
It follows from \cite[Lemmas 3.1,3.2]{CLZ} that the solution of problem \eqref{subso}, denoted by $v$,
satisfies
\begin{equation}\label{vto1}
\lim_{t\to\infty}v(t,x)=V^*(x)\ \ \mbox{ locally uniformly in\ $x\in[r,\infty)$},
\end{equation}
where $V^*$ is the unique solution of
\[
v''+g(v)=0<v\ \ \mbox{ in } (r,\infty),\ \ \  v(r)=0,\ \ \ v(\infty)=1.
\]
In view of the condition {\bf{(H)}}, a direct application of the comparison principle deduces that $u_\sigma(t,x)\geq v(t,x)$ for $t>0$ and $x\geq r$.
This, together with \eqref{vto1} and Theorem \ref{thm:convergence}, implies that spreading must happen for $u_\sigma$.
\end{proof}

Define
\[
\Sigma_1=\big\{\sigma>0:\ u_\sigma(t,x)\to1\ \ \mbox{locally uniformly on} \ \R_+,\ \mbox{as}\ t\to\infty\}.
\]
In the sequel, we will show that the set $\Sigma_1$ is open. That is, we have
\begin{lem}\label{lem:sd112} For any $L>0$, $\Sigma_1$ is a nonempty open interval.
\end{lem}
\begin{proof} If $\sigma_1\in \Sigma_1$, then spreading happens for $\sigma=\sigma_1$. For any $\alpha\in(\theta^*,1)$
and $l_\alpha$ given in \eqref{la1},  we can find $T_1>0$ large enough such that
\begin{equation}
\label{u-v_Z}
 u_{\sigma_1}(T_1, x)>\alpha \mbox{ in } [L,L+l_\alpha],
\end{equation}
where $u_{\sigma_1}$ is the solution of \eqref{p} with $u_0=\phi_{\sigma_1}$. By the continuous dependence of the solution on initial values,
we can find a small $\epsilon>0$ such that the solution $u_{\sigma_1-\epsilon}$ of \eqref{p} with $u_0=\phi_{\sigma_1-\epsilon}$ satisfies \eqref{u-v_Z}. It then follows from Lemma \ref{lem:spreadc} that spreading happens for $u_{\sigma_1-\epsilon}$, which infers that $\sigma_1-\epsilon\in\Sigma_1$.
On the other hand, the comparison principle implies that $\sigma\in \Sigma_1$ for any $\sigma>\sigma_1$. Thus  $\Sigma_1$ is an open interval.

Next, we show that $\Sigma_1$ is non-empty. As $f,\,g$ are globally Lipschitz on $[0,\infty)$, we can find $M>0$ such that
 \[
 f(u),\  g(u)\geq -Mu,\ \ \mbox{ for all}\ u\geq 0.
 \]

Let us consider the following problem
\begin{equation}\nonumber
\left\{
\begin{array}{ll}
 w_t=w_{xx}-Mw, & t>0,\ x\in\R,\\
 w_x(t,0)=0, &  t>0,\\
 w(0,x)=\phi_\sigma(|x|),& x\in[-\hbar,\hbar],\\
 w(0,x)=0,& x\not\in[-\hbar,\hbar].
\end{array}
\right.
\end{equation}
Clearly, this problem admits a unique positive solution $w$. The comparison principle can be
used to deduce that, for $t\geq0,\,x\geq0$,
\[
u_\sigma(t,x)\geq w(t,x)=\int_{-\hbar}^\hbar\frac{e^{-\frac{(x-y)^2}{4t}-Mt}}{\sqrt{4\pi t}}\phi_\sigma(|y|)dy.
\]
Then for any $\alpha\in(\theta^*,1)$, we have $u_\sigma(1,x)>\alpha$ in $[L,L+2l_\alpha]$ provided that $\sigma$ is sufficiently large. This and Lemma \ref{lem:spreadc} yield that $\sigma\in\Sigma_1$. \end{proof}

\begin{lem}\label{lem:LL1}
If there exists $\bar{L}>0$ such that problem \eqref{p} with $L=\bar{L}$ has a ground state, then
problem \eqref{p} with any $L\in(0,\bar{L}]$ admits a ground state. Henceforth, problem \eqref{p} admits a ground state for any $0<L<L^*$, where $L^*$ is defined in \eqref{L8}.
\end{lem}

\begin{proof} Suppose by way of contradiction that there is $\bar{L}_0\in(0,\bar{L})$ such that problem \eqref{p}
with $L=\bar{L}_0$ does not have a ground state. We claim that vanishing does not happen for problem \eqref{p} with $L=\bar{L}_0$. Otherwise, by Lemmas \ref{lem:sd0} and \ref{lem:sd112}, both $\Sigma_0$ and $\Sigma_1$ are nonempty open intervals for problem \eqref{p} with $L=\bar{L}_0$. This, together with Theorem \ref{thm:convergence} and Lemma \ref{lem:staso} that
problem \eqref{p} with $L=\bar{L}_0$ has a ground state, leads to a contradiction. Thus our claim holds.

As a consequence, only spreading can happen for problem \eqref{p} with $L=\bar{L}_0$. The comparison principle, combined with the condition {\bf{(H)}}, further yields that only spreading can happen for problem \eqref{p} with $L=\bar{L}$. This clearly contradicts with our assumption that problem \eqref{p} with $L=\bar{L}$ has a ground state.

By the definition of $L^*$, problem \eqref{p} admits a ground state for any $0<L<L^*$. \end{proof}

In the sequel, we aim to give an upper-bound estimate for $L^*$, which is stated as follows.

\begin{prop}\label{propr1}
Let $L^*$ be given in \eqref{L8}. Then the following estimate holds:
\[
L^*\leq\int_0^{\theta^*}\frac{1}{\sqrt{2\int_r^{\theta^*}f(s)ds}}dr<\infty.
\]
\end{prop}
\begin{proof}
Let us consider the following auxiliary problem:
\begin{equation}\label{prob-q}
\left\{
\begin{array}{ll}
q'' +f(q)=0,\ q(x)>0, \ \  x\in[0,l),\\
q'(0)=q(l)=0.
\end{array}
\right.
\end{equation}
It is well known that a positive solution of problem \eqref{prob-q} (if exists) must be unique and be decreasing on $[0,l]$.

Using $q'$ to denote $\frac{dq}{dx}$, we can rewrite the first equation in \eqref{prob-q} into the equivalent form
\begin{equation}\label{eq-qp}
q'=p,\ \ \ p'=-f(q),
\end{equation}
or,
\begin{equation}\label{pq}
\frac{dp}{dq} = - \frac{f(q)}{p},\quad \mbox{when } p\not= 0.
\end{equation}
The positive solution of \eqref{pq} with
$p(q)|_{q=\theta^*} =0$ is given explicitly by
\begin{equation}\label{p0-left}
p(q) = -\sqrt{2\int_q^{\theta^*} f(s)ds}, \quad \mbox{for }
q\in [0,\theta^*].
\end{equation}

The positive solution $p(q)$ ($q\in [0, \theta^*]$) corresponds to a trajectory $(q(x),p(x))$ of \eqref{eq-qp}
that connects $(\theta^*, 0)$ and $(0,-\omega_0)$ in the $qp$-plane, where $\omega_0=\sqrt{2\int_0^{\theta^*}f(s)ds}$.
In order to give a solution $q$ of \eqref{prob-q}, we may assume that it passes through $(0,-\omega_0)$ at $l\in (0,\infty]$
and approaches $(\theta^*,0)$ as $x$ goes to $0$. Then using \eqref{eq-qp} and \eqref{p0-left}
we can easily deduce
\begin{equation}\label{q-eq}
L^0:=\int_0^{\theta^*} \frac{1}{\sqrt{2 \int_r^{\theta^*} f(s) ds}}dr
= \int_0^{\theta^*} \frac{1}{\sqrt{\omega_0^2 - 2\int_0^r f(s) ds}}dr<\infty,
\end{equation}
which means that when $l=L^0$, problem \eqref{prob-q} admits a unique positive solution $q(x)$ with $q(0)=\theta^*$.

By a similar argument as above, we can prove that when $l>L^0$, problem \eqref{prob-q} also has a unique positive solution $q_{l}(x)$ satisfying $q_l(0)>\theta^*$. Moreover, it is easy to check that if $L^0\leq l_1< l_2$, then $q_{l_1}(x)< q_{l_2}(x)$ for $x\in[0,l_1]$.

Now we prove $L^*\leq L^0$ by contradiction. Suppose that $L^*> L^0$. For any fixed $L>L^0$, consider the following problem
\begin{equation}\label{subp-q1}
\left\{
\begin{array}{ll}
v_t=v_{xx} +f(v), &  t>0,\ x\in (0,L),\\
v_x(t,0)=v(t, L)=0, & t>0,\\
v(0,x)=v_0(x), & x\in [0,L].
\end{array}
\right.
\end{equation}
Some standard analysis shows that for any $v_0(x)\geq,\not\equiv 0$, then the unique positive solution $v(t,x)$ of
\eqref{subp-q1} satisfies
\begin{equation}\label{vtoq}
\|v(t,\cdot)- q_L(\cdot)\|_{L^\infty([0, L])} \to 0,\quad \mbox{as }\, t\to\infty,
\end{equation}
where $q_L$ is the unique positive solution of \eqref{prob-q} with $l=L$. The comparison principle implies that the solution $u(t,x)$
of \eqref{p} where the initial datum $u_0\in X^+$ can be chosen so that $u_0(x)\geq v_0(x)\geq,\not\equiv 0$ in $[0,L]$ satisfies
\[
u(t,x)\geq v(t,x),\ \ \mbox{ for all}\ t>0,\ x\in[0,L].
\]
Combining this, the fact that $q_L(0)>\theta^*$ with $L>L^0$, and  Lemma \ref{lem:staso}, we see that $u(t,x)\to 1$ locally
uniformly in $[0,\infty)$ as $t\to\infty$, which means that only spreading can happen for $u$. That is, problem \eqref{p} does
not have a ground state for any $L>L^0$. This contradicts the definition of $L^*$, and so $L^*\leq L^0$. Thus the
proof is complete.
\end{proof}

We end this subsection with the following spreading result.
\begin{lem}\label{lem:spread}
Let $L^*$ be given in \eqref{L8}.  Then spreading must happen for problem \eqref{p} with $L>L^*$.
\end{lem}
\begin{proof}
Since $L^* \geq L_*$, then vanishing does not happen for problem \eqref{p}
with $L>L^*$ due to Lemma \ref{lem:sd0}(ii). The definition of $L^*$, together with
Theorem \ref{thm:convergence} and Lemma \ref{lem:staso}, implies that only spreading can happen for problem \eqref{p} with
$L>L^*$, which completes the proof of this lemma.
\end{proof}

\subsection{Proof of Theorem \ref{thm:dybe}}
Based on the preparation of the previous subsections,
we are now ready to give

\smallskip

\noindent
 {\bf Proof of Theorem \ref{thm:dybe}:}\ (I) When $0<L<L_*$, it follows from Lemmas \ref{lem:sd0} and \ref{lem:sd112}
 that $\Sigma_0$ is an open interval $(0,\sigma_*)$ with $\sigma_*:=\sup \Sigma_0\in(0,\infty)$, and that
 $\Sigma_1$ is an open interval $(\sigma^*,\infty)$ with $\sigma^*:=\inf \Sigma_1\in(0,\infty)$.
 Clearly, $\sigma_*\leq \sigma^*$. Then, Theorem \ref{thm:convergence} and Lemma
 \ref{lem:staso} allows us to assert that each solution $u_\sigma(t,x)$ with $\sigma\in [\sigma_*, \sigma^*]$  is a transition one.

\smallskip

(II) Let us assume that $L_*<L<L^*$. It follows from Lemma \ref{lem:sd0} that $\Sigma_0$ is empty in this case, which means
 that vanishing does not happen for any $\sigma>0$. On the other hand, from Lemma \ref{lem:sd112}, it is known that $\Sigma_1$ is
a nonempty open interval $(\sigma^*,\infty)$ with $\sigma^*:=\inf \Sigma_1$. In addition, by Lemma \ref{lem:LL1}, problem \eqref{p} admits a ground state for any $0<L<L^*$. This, combined with Theorem \ref{thm:convergence}, shows that each solution $u(t,x)$ with $\sigma\in (0, \sigma^*]$ is a transition one.

 \smallskip

(III) The conclusion for the case $L>L^*$ follows from Lemma \ref{lem:spread} immediately.

 \smallskip
The whole proof of Theorem \ref{thm:dybe} is thus complete. {\hfill $\Box$}

\section{A Separate Protection Zone: proof of Theorem \ref{thm:dybe1}}\label{sect:multiple case}
In this section, we deal with system \eqref{q} and establish its dynamics (i.e., Theorem \ref{thm:dybe1})
in the same spirit as that of proving Theorem \ref{thm:dybe}.

We shall begin with the following eigenvalue problem:
\begin{equation}\label{eigen-q}
\left\{
 \begin{array}{ll}
 - \varphi'' -f'(0)\varphi =\lambda \varphi, & x \in(L_1, L_2),\\
 - \varphi'' -g'(0)\varphi =\lambda \varphi, & x\in(0, L_1)\cup(L_2,\infty),\\
 \varphi'(0)=\varphi(\infty)=0, \\
 \varphi(L_i-0) = \varphi( L_i+0), & i=1, 2, \\
 \varphi'(L_i-0) = \varphi'( L_i+0), & i=1, 2.
 \end{array}
 \right.
\end{equation}

As in Subsection 3.1, we consider the eigenvalue problem on $\R$ associated with \eqref{eigen-q}:
 \begin{equation}\label{eigen-q1}
 \left\{
 \begin{array}{ll}
 - \varphi'' +\tilde h(x)\varphi =\lambda \varphi, & -\infty<x<\infty,\\
 \varphi(-\infty) = \varphi(\infty)=0,
 \end{array}
 \right.
 \end{equation}
where
 $$
\tilde h(x) = \left\{
\begin{array}{ll}
     -f'(0),\ \ & x\in[L_1, L_2]\cup[-L_2, -L_1],\\
      -g'(0), \ \ & x\in\R\setminus{([L_1, L_2]\cup[-L_2, -L_1])}.
\end{array}
\right.
$$

Then, it is well known that \eqref{eigen-q1} and \eqref{eigen-q} have the same principal eigenvalue,
denoted by $\tilde \lambda_1 (L)$. Furthermore, by \cite[Proposition 6.11]{BHR} (or \cite[Theorem 4.1]{BR}),
it holds that
 \begin{equation}\nonumber
 \mbox{$\tilde\lambda_1^R(L)$\ is decreasing in\ $R>0$}\ \, \mbox{and}\ \,\lim_{R\to\infty}\tilde\lambda_1^R(L)\leq\tilde\lambda_1 (L),
 \end{equation}
where $\tilde\lambda_1^R(L)$ is the principal eigenvalue of
 \begin{equation}\nonumber
 \left\{
 \begin{array}{ll}
 - \varphi'' +\tilde h(x)\varphi =\lambda \varphi, & -R<x<R,\\
 \varphi(-R) = \varphi(R)=0.
 \end{array}
 \right.
 \end{equation}

We have the following result.
\begin{lem}\label{lem:1eigenvaluemp2}
For any given $0<L_1<L_2$, let $L=L_2-L_1$ and $\tilde{\lambda}_1(L)$
be the principal eigenvalue of \eqref{eigen-q}. Then we have
\[
\tilde{\lambda}_1(L)\in(-f'(0),-g'(0)),
\]
and
\[
L=\frac{1}{\theta_2}\Big\{ \arctan\Big[\frac{\theta_1}{\theta_2}\cdot\frac{e^{\theta_1 L_1}
-e^{-\theta_1 L_1}}{e^{\theta_1 L_1}+e^{-\theta_1 L_1}}\Big]+  \arctan \frac{\theta_1}{\theta_2}\Big\},
\]
where
\[
\theta_1=\sqrt{-(g'(0)+\tilde{\lambda}_1(L))},\ \ \ \theta_2=\sqrt{f'(0)+\tilde{\lambda}_1(L)}.
\]
Moreover, $\tilde{\lambda}_1(L)$ is decreasing with respect to $L>0$, and there exists a unique
$\tilde{L}_*>L_*$ such that $\tilde{\lambda}_1(L)<0$ if $L>\tilde{L}_*$, $\tilde{\lambda}_1(L)=0$ if $L=\tilde{L}_*$, and $\tilde{\lambda}_1(L)>0$ if $0<L<\tilde{L}_*$.
\end{lem}

\begin{proof} Let us set $\tilde{\lambda}_1=\tilde{\lambda}_1(L)$ for simplicity. The similar analysis as that of Lemma \ref{lem:1eigenvalue}
shows that $\tilde{\lambda}_1\in(-f'(0),-g'(0))$ for any $0<L_1<L_2<\infty$.

Since $\tilde{\lambda}_1\in(-f'(0),-g'(0))$, then
\begin{equation}\label{thee}
0<\theta_i<\sqrt{f'(0)-g'(0)},\ \ \mbox{ for }\, i=1,2.
\end{equation}

For $x\in[0,L_1)$, it follows from the second equation of  \eqref{eigen-q} that
\[
-\varphi''=(\tilde{\lambda}_1+g'(0))\varphi\ \ \mbox{with}\  g'(0)+\tilde{\lambda}_1<0,
\]
which implies that there are two constant $\tilde{C}_1$ and $\tilde{C}_2$ such that
\[
\varphi(x)=\tilde{C}_1e^{\theta_1 x}+\tilde{C}_2e^{-\theta_1 x},\ \ \forall x\in[0,L_1).
\]
This, together with $\varphi'(0)=0$, yields that $\tilde{C}_1=\tilde{C}_2>0$, and so for  $x\in[0,L_1)$,
\[
\varphi(x)=\tilde{C}_1\big(e^{\theta_1 x}+e^{-\theta_1 x}\big),
\]
and
\[
\varphi'(x)=\tilde{C}_1\theta_1 \big(e^{\theta_1 x}-e^{-\theta_1 x}\big)>0.
\]
Hence we have
\begin{equation}\label{eeqq1}
\frac{\varphi'(L_1-0)}{\varphi(L_1-0)}=\theta_1\cdot\frac{e^{\theta_1 L_1}
-e^{-\theta_1 L_1}}{e^{\theta_1 L_1}+e^{-\theta_1 L_1}}.
\end{equation}
Similarly, for $x\in(L_2,\infty)$, there are two constant $\tilde{C}_3$ and $\tilde{C}_4$ such that
\[
\varphi(x)=\tilde{C}_3e^{\theta_1 x}+\tilde{C}_4e^{-\theta_1 x},\  \ \forall x\in(L_2,\infty).
\]
Using $\varphi(\infty)=0$, we infer that $\tilde{C}_4>0=\tilde{C}_3$, and so for $x\in(L_2,\infty)$,
\[
\varphi(x)=\tilde{C}_4e^{-\theta_1 x},
\]
and
\[
\varphi'(x)=-\tilde{C}_4\theta_1 e^{-\theta_1 x}<0.
\]
Then we obtain
\begin{equation}\label{eeqq2}
\frac{\varphi'(L_2+0)}{\varphi(L_2+0)}=-\theta_1.
\end{equation}

Since $\varphi'(L_i+0)=\varphi'(L_i-0)$ $(i=1,2)$, it follows that $\varphi'(L_1+0)>0>\varphi'(L_2-0)$.
Therefore, there is a constant $a\in(L_1,L_2)$ such that $\varphi'(a)=0$.

We further claim that there exists a unique $a\in(L_1,L_2)$ fulfilling $\varphi'(a)=0$. Otherwise there are two constants $a_i\in (L_1,L_2)$ with $a_1<a_2$ satisfying $\varphi'(a_i)=0$ $(i=1,2)$. It follows from the first equation of  \eqref{eigen-q} that
\[
-\varphi''=(f'(0)+\tilde{\lambda}_1)\varphi, \ \ \forall x\in[a_1,a_2].
\]
Integrating the above equation on $[a_1,a_2]$, we deduce that
\[
(f'(0)+\tilde{\lambda}_1)\int_{a_1}^{a_2}\varphi(x)dx=0,
\]
which gives $f'(0)+\tilde{\lambda}_1=0$, a contradiction!

Now, when $x\in(L_1,L_2)$, we get from the second equation of \eqref{eigen-q} that
\[
-\varphi''=(f'(0)+\tilde{\lambda}_1)\varphi\  \ \mbox{with}\ \, f'(0)+\tilde{\lambda}_1>0.
\]
Similarly, we can find two constants $\tilde{C}_5$ and $\tilde{C}_6$  such that
\[
\varphi(x)=\tilde{C}_5\cos [\theta_2(x-a)] +\tilde{C}_6\sin [\theta_2(x-a)], \ \ \forall x\in[L_1,L_2].
\]
Since $\varphi'(a)=0$ and $\varphi>0$ for $x\geq 0$, then $\tilde{C}_5>0=\tilde{C}_6$. In turn, it holds
\begin{equation}\label{pt2}
\varphi(x)=\tilde{C}_5\cos [\theta_2(x-a)], \ \ \forall x\in[L_1,L_2].
\end{equation}
Moreover, basic computation gives that
\[
\frac{\varphi'(L_1+0)}{\varphi(L_1+0)}=-\theta_2\tan[\theta_2(L_1-a)],\ \
\frac{\varphi'(L_2-0)}{\varphi(L_2-0)}=-\theta_2\tan[\theta_2(L_2-a)].
\]

By virtue of \eqref{eeqq1} and \eqref{eeqq2}, it then follows that
\begin{equation}\label{eeqq3}
\frac{\theta_1}{\theta_2}\cdot\frac{e^{\theta_1 L_1}
-e^{-\theta_1 L_1}}{e^{\theta_1 L_1}+e^{-\theta_1 L_1}}
=\tan[\theta_2(a-L_1)]>0
\end{equation}
and
\begin{equation}\label{eeqq4}
\frac{\theta_1}{\theta_2}=\tan[\theta_2(L_2-a)]>0.
\end{equation}
In view of \eqref{pt2}, \eqref{eeqq3} and \eqref{eeqq4}, we may assume that
\begin{equation}\label{pt1}
0<\theta_2(L_2-a)<\frac{\pi}{2},\ \ 0<\theta_2(a-L_1)<\frac{\pi}{2}.
\end{equation}

Due to
\[\frac{e^{\theta_1 L_1}
-e^{-\theta_1 L_1}}{e^{\theta_1 L_1}+e^{-\theta_1 L_1}}<1,
\]
it is easily seen from \eqref{eeqq3} and \eqref{eeqq4} that $a-L_1<L_2-a$, which means that
\[
a\in\big(L_1,\frac{L_1+L_2}{2}\big).
\]
Consequently, $a\to L_1$ as $L_2\to L_1$. Using this fact and \eqref{eeqq4} again, we know
that
\begin{equation}\label{Lto1}
\tilde{\lambda}_1\to-g'(0)>0,\ \ \mbox{as}\ L_2\to L_1.
\end{equation}

Next we determine the limit of $\tilde{\lambda}_1$ in the case that $L_1$ is fixed and $L_2\to\infty$.
Suppose that there is $a^*\in(L_1,\infty)$ such that $a\to a^*$ as $L_2\to \infty$. It then follows from
\eqref{pt1} that $\theta_2\to 0$ and
$\tilde{\lambda}_1\to-f'(0)$ in this case, which contradicts with \eqref{eeqq3}. On the other hand, suppose
$a\to L_1$ as $L_2\to \infty$, we get from \eqref{eeqq3} and \eqref{thee} that
\[
\theta_1\to 0,\ \ \tilde{\lambda}_1\to-g'(0)\ \ \mbox{ and }\ \ \theta_2\to \sqrt{f'(0)-g'(0)}.
\]
A contradiction occurs due to \eqref{eeqq4}.

Consequently, we have proved that
\[
a\to \infty,\ \ \mbox{ as}\ L_2\to\infty.
\]
Combining this, \eqref{eeqq3}, \eqref{eeqq4} and \eqref{pt1}, as $L_2\to \infty$, one can easily see that
\begin{equation}\label{Lto2}
\tilde{\lambda}_1\to-f'(0)<0
\end{equation}
and
\[
\theta_2\to 0,\ \ \ \theta_2(a-L_1)\to\frac{\pi}{2}\ \ \mbox{ and }\ \ \theta_2(L_2-a)\to\frac{\pi}{2}.
\]

Furthermore, making use of \eqref{eeqq3}, \eqref{eeqq4} and \eqref{pt1} again, we deduce that
\[
\theta_2(a-L_1)=\arctan\Big[\frac{\theta_1}{\theta_2}\cdot\frac{e^{\theta_1 L_1}
-e^{-\theta_1 L_1}}{e^{\theta_1 L_1}+e^{-\theta_1 L_1}}\Big]\ \ \mbox{ and }\ \
\theta_2(L_2-a)=\arctan \frac{\theta_1}{\theta_2}.
\]
Adding these two identities infers
\begin{equation}\label{LL11}
L=L_2-L_1=\frac{1}{\theta_2}\Big\{ \arctan\Big[\frac{\theta_1}{\theta_2}\cdot\frac{e^{\theta_1 L_1}
-e^{-\theta_1 L_1}}{e^{\theta_1 L_1}+e^{-\theta_1 L_1}}\Big]+  \arctan \frac{\theta_1}{\theta_2}\Big\}.
\end{equation}

It is noted that $\theta_1$ is decreasing while $\theta_2$ is increasing with respect to $\tilde{\lambda}_1$.
By virtue of \eqref{LL11}, some basic analysis shows that $\tilde{\lambda}_1$ is decreasing with respect to $L>0$.
In addition, by \eqref{Lto1} and \eqref{Lto2}, we conclude that there is a unique $\tilde{L}_*>0$ such that $\tilde{\lambda}_1(L)<0$ if $L>\tilde{L}_*$, $\tilde{\lambda}_1(L)=0$ if $L=\tilde{L}_*$ and $\tilde{\lambda}_1(L)>0$ if $0<L<\tilde{L}_*$.

It remains to show that $\tilde{L}_*>L_*$. In fact, we obtain from \eqref{LL11} that
\begin{equation}\nonumber
L>\frac{1}{\theta_2}\arctan \frac{\theta_1}{\theta_2},\ \ \forall L>0.
\end{equation}
Thus, by the definition of $\tilde{L}_*,\,\theta_1$ and $\theta_2$, it holds
\begin{equation}\nonumber
\tilde{L}_*>\frac{1}{\sqrt{f'(0)}}\arctan\sqrt{-\frac{g'(0)}{f'(0)}}=L_*.
\end{equation}
The proof is now complete. \end{proof}

With the aid of Theorem \ref{thm:convergence} and Lemma \ref{lem:1eigenvaluemp2}, we can use the similar argument as in Section 3 to prove
Theorem \ref{thm:dybe1}. The details are omitted here.

\smallskip
Finally, we derive the estimate \eqref{L8-2} for $\tilde{L}^*$. That is, we have

\begin{prop}\label{propr1-1}
Let $\tilde{L}^*$ be given in \eqref{L8-1}, then \eqref{L8-2} holds.
\end{prop}
\begin{proof} The analysis is similar to that of Proposition \ref{propr1}. Instead of \eqref{prob-q}, we here appeal to
the following auxiliary problem:
\begin{equation}\label{prob-q2}
\left\{
\begin{array}{ll}
q'' +f(q)=0,\ q(x)>0, \ \ \  x\in (L_1,L_2),\\
q(L_1)=q(L_2)=0.
\end{array}
\right.
\end{equation}
By the proof of Proposition \ref{propr1}, we know that when $l:=\frac{L_2-L_1}{2}=L^0$,
where $L^0$ is given in \eqref{q-eq}, problem \eqref{prob-q2} possesses a unique positive solution $q$, which is increasing on $[L_1,\frac{L_2+L_1}{2}]$
while is decreasing on $[\frac{L_2+L_1}{2},L_2]$ and satisfies $q(\frac{L_2+L_1}{2})=\theta^*,\, q'(\frac{L_2+L_1}{2})=0$.
In addition, when $l>L^0$, problem \eqref{prob-q2} also has a unique positive solution $q_{l}(x)$ fulfilling $q_l'(\frac{L_2+L_1}{2})=0,\,q_l(\frac{L_2+L_1}{2})>\theta^*$.

Now following the argument of Proposition \ref{propr1} with some obvious modification, we can conclude
that $\tilde{L}^*\leq 2L^0$.
\end{proof}

\section{Conclusion}

More than 99 percent of all species amounting to over five billion species, that ever lived on Earth are estimated to be extinct; most species that become extinct are never scientifically documented \cite{Ne}. Some scientists estimate that up to half of presently existing plant and animal species may become extinct by 2100 \cite{Wi}. Humans can cause extinction of a species through overharvesting, pollution, habitat destruction, introduction of invasive species (such as new predators and food competitors), overhunting, and other influences.

More and more environmental groups and governments are concerned with the extinction of species caused by humanity, and they are trying  to prevent further extinctions through a variety of conservation programs. It is recognized that protection zones provide many economic, social, environmental, and cultural values. The role of protection zone in preventing population from extinction has been investigated in \cite{CLMS,CSW,DuLiang,DPW,DuS2,DuS3,HZ,HZ2,LW,LWL,LY,Oe,WL,ZZL} and the references therein for reaction-diffusion models; we note that bounded habitats are assumed in those works.

In the present work, we have been concerned with a reaction-diffusion model with strong Allee effect, in which the endangered single species lives in an entire one-dimensional space; when the initial density of the species is small, it is known that the species will die out in the long run. In order to save such endangered species, we have introduced a bounded protection zone, within which the species growth is governed by the classical Fisher-KPP nonlinear reaction.

Assume that the total length of protection zone is fixed as $2L$, and the initial data are symmetric with respect to the origin. Regarding the design of the protection zone, we have indeed proposed the following two different scenarios (refer to Figure \ref{fig:a}):

\smallskip
\ \ \ (i)\ {\it the protection zone is $(-L,L)$ and thus is connected};

\ \ \,(ii)\ {\it the protection zone is $(-L_2,-L_1)\cup(L_1,L_2)$ with $0<L_1<L_2$ and thus is separate.}
(Recall that the solution of the problems under consideration is symmetric with respect to the origin).

\smallskip Our results (Theorems \ref{thm:dybe} and \ref{thm:dybe1}) have shown that
in each scenario, there are two critical values $0<\hat L_*\leq\hat L^*$, and proved that a vanishing-transition-spreading trichotomy result holds when the length $2L$ of protection zone is smaller than $2\hat L_*$ (that is, $0<L<\hat L_*$); a transition-spreading dichotomy result holds when $\hat L_*<L<\hat L^*$; only spreading happens when $L>\hat L^*$. As a consequence, our results suggest that the protection zone works only when its length $2L$ is larger than the critical value $2\hat L_*$. Here,
$(\hat L_*,\hat L^*)=(L_*,L^*)$ in scenario (i) (see Theorem \ref{thm:dybe}) and $(\hat L_*,\hat L^*)=(\tilde L_*,\tilde L^*)$ in scenario (ii) (see Theorem \ref{thm:dybe1}).

Furthermore, in light of Lemma \ref{lem:1eigenvaluemp2}, it holds
 $$L_*<\tilde L_*.$$
This, combined with Theorems \ref{thm:dybe} and \ref{thm:dybe1}, enables us to further conclude that the scenario that
the protection zone is designed into a connected region (Figure \ref{fig:a}: Left) is better for species spreading than the scenario that the protection zone is designed into a separate region consisting of two parts (Figure \ref{fig:a}: Right).
In this sense, we conjecture that a connected protection zone is the optimal one once the total length (regarded as resource) is given, even if the initial data may not be symmetric.

In the current paper, we have assumed that the species live in an entire one-dimensional space.
Nevertheless, the habitat of a biological population, in general, can be rather complicated. For example,
natural river systems are often in a spatial network structure such as dendritic
trees. The network topology (i.e., the topological structure of a river network) can greatly influence the species persistence and extinction. Therefore, as in \cite{JPS,Ra,SMA}, it would be interesting to consider a more general river habitat (bounded or unbounded) consisting of more than one branch. Moreover, if a branch is bounded, the works in \cite{LLL,LLL2,LL} have shown that different boundary conditions could be vital in the population dynamics. We plan to study these problems with Allee effect and protection zone in future work.

\begin{figure}[htbp]
\centering
\includegraphics[width=7cm]{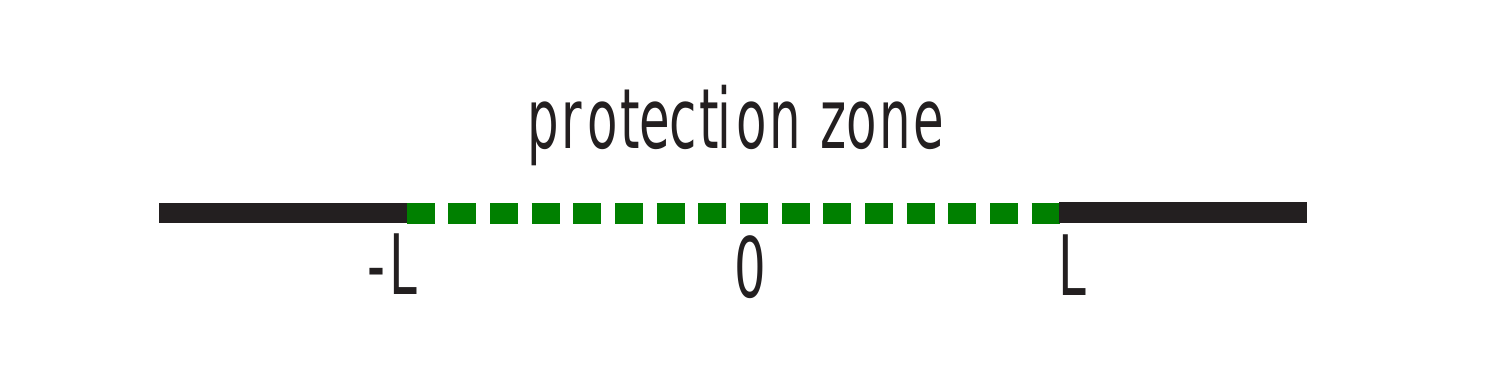}
\includegraphics[width=7cm]{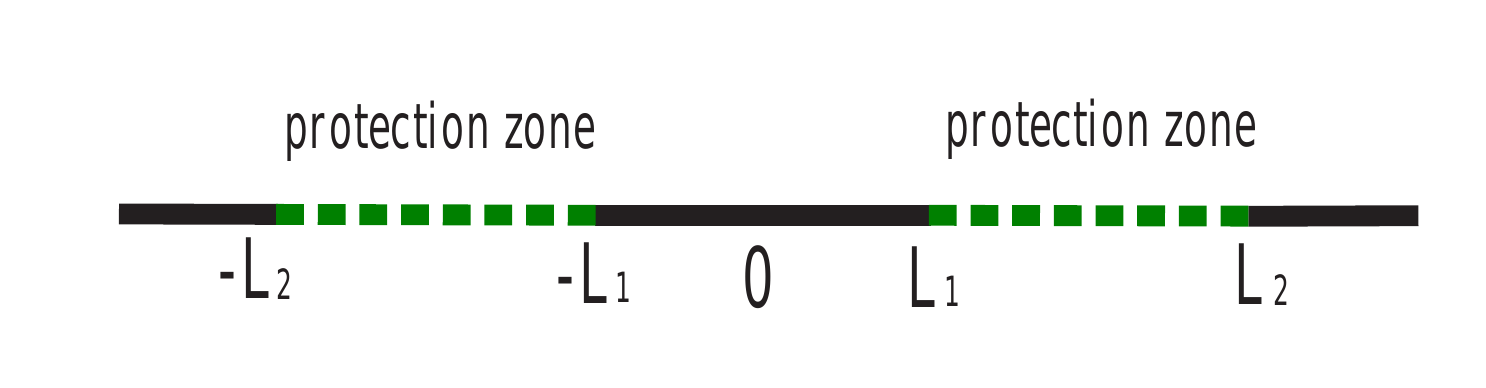}
\caption{{\small\ Left: The connected protection zone corresponding to system \eqref{p};
Right: The separate protection zone corresponding to system \eqref{q}.}}\label{fig:a}
\end{figure}

\section*{Acknowledgements}
The authors would like to thank Dr. Feng Zhou (Shandong Normal University, China) and Dr. Maolin Zhou (University of New England, Australia) for some valuable discussions during the preparation of the paper.

\end{document}